\documentclass[12pt,reqno]{amsart}

\usepackage{fancyvrb}

\usepackage{amssymb, amsmath, amsthm}
\usepackage[bookmarks, colorlinks=true, linkcolor=blue, citecolor=blue, urlcolor=blue]{hyperref}
\usepackage[alphabetic,lite]{amsrefs}
\usepackage{verbatim}
\usepackage{amscd}   
\usepackage[all]{xy} 
\usepackage{youngtab} 
\usepackage{young} 
\usepackage{ytableau} 
\usepackage{tikz}
\usepackage{ mathrsfs }
\usepackage{cases}
\usepackage{verbatimbox}
\usepackage{tabu}
\usepackage{multicol}

\usepackage[margin=1in]{geometry}

\newcommand{\Seg}{\operatorname{Seg}}

\newcommand{\defi}[1]{{\upshape\bf \sffamily #1}}

\renewcommand{\a}{\alpha}

\newcommand{\bw}{\bigwedge}
\renewcommand{\det}{\textrm{det}}
\renewcommand{\ll}{\lambda}

\newcommand{\oo}{\otimes}

\newcommand{\R}{\overline{R}}

\renewcommand{\S}{\overline{S}}

\newcommand{\lexrk}{\operatorname{lexrk}}
\newcommand{\rk}{\operatorname{rk}}

\newcommand{\GL}{\operatorname{GL}}

\newcommand{\Sym}{\operatorname{Sym}}
\newcommand{\Tor}{\operatorname{Tor}}

\newcommand{\bb}[1]{\mathbb{#1}}

\newcommand{\op}[1]{\operatorname{#1}}
\renewcommand{\rm}[1]{\textrm{#1}}
\newcommand{\mc}[1]{\mathcal{#1}}

\newcommand{\ol}[1]{\overline{#1}}

\newcommand{\ul}[1]{\underline{#1}}

\def\lra{\longrightarrow}

\newtheorem{theorem}{Theorem}[section]
\newtheorem*{theorem*}{Theorem}
\newtheorem{lemma}[theorem]{Lemma}

\newtheorem{corollary}[theorem]{Corollary}
\newtheorem*{corollary*}{Corollary}

\newtheorem*{main*}{Main Theorem}

\theoremstyle{definition}
\newtheorem{defn}[equation]{Definition}

\newtheorem{eg}[theorem]{Example}
\newenvironment{example}[1][]{\begin{eg}[#1] \pushQED{\qed}}{\popQED \end{eg}}

\theoremstyle{remark}
\newtheorem{rmk}[theorem]{Remark}
\newenvironment{remark}[1][]{\begin{rmk}[#1] \pushQED{\qed}}{\popQED \end{rmk}}

\numberwithin{equation}{section}

\def\blue#1{{\textcolor{blue}{#1}}}
\def\green#1{{\textcolor{green}{#1}}}
\def\yellow#1{{\textcolor{yellow}{#1}}}

\newcommand{\luke}[1]{{\color{green} \sf $\clubsuit\clubsuit\clubsuit$ Luke: [#1]}}

\begin{document}

\title{On the (non-)vanishing of syzygies of Segre embeddings}

\author{Luke Oeding}
\address{Department of Mathematics, Auburn University, 221 Parker Hall, Auburn, AL 36849}
\email{oeding@auburn.edu}

\author{Claudiu Raicu}
\address{Department of Mathematics, University of Notre Dame, 255 Hurley, Notre Dame, IN 46556\newline
\indent Institute of Mathematics ``Simion Stoilow'' of the Romanian Academy}
\email{craicu@nd.edu}

\author{Steven V Sam}
\address{Department of Mathematics, University of Wisconsin, Madison}
\email{svs@math.wisc.edu}


\date{\today}


\begin{abstract}
We analyze the vanishing and non-vanishing behavior of the graded Betti numbers for Segre embeddings of products of projective spaces. We give lower bounds for when each of the rows of the Betti table becomes non-zero, and prove that our bounds are tight for Segre embeddings of products of $\bb{P}^1$. This generalizes results of Rubei concerning the Green--Lazarsfeld property $N_p$ for Segre embeddings. Our methods combine the Kempf--Weyman geometric technique for computing syzygies, the Ein--Erman--Lazarsfeld approach to proving non-vanishing of Betti numbers, and the theory of algebras with straightening laws.
\end{abstract}

\maketitle

\section{Introduction}\label{sec:intro}

For non-negative integers $a_1,\cdots,a_n$, let $X=\rm{Seg}(a_1,\cdots,a_n)$ denote the image of the $n$-Segre embedding
\[\bb{P}^{a_1}\times\cdots\times\bb{P}^{a_n}\lra\bb{P}^A,\]
where $A=(a_1+1)\cdots(a_n+1)-1$ and the ground field is the complex numbers $\bb{C}$. Let $S$ denote the homogeneous coordinate ring of $\bb{P}^A$, let $I\subset S$ denote the ideal of equations vanishing on $X$, and let $R=S/I$ denote the homogeneous coordinate ring of $X$. We study the (non-)vanishing behavior of the syzygies of $X$, defined by:
\[K_{p,q}(a_1,\cdots,a_n)=\Tor_p^S(R,\bb{C})_{p+q}.\]

\begin{main*} Assume that $a_1\geq\cdots\geq a_n\geq 0$. We have the non-vanishing for $K_{p,1}$ groups
\begin{equation}\label{eq:Kp1-not-0}
 K_{p,1}(a_1,\cdots,a_n) \neq 0 \mbox{ for } 1\leq p \leq (a_1+1)\cdots(a_{n-1}+1) + a_n - 2.
\end{equation}
For arbitrary $q$ we have the vanishing of $K_{p,q}$ groups
 \begin{equation}\label{eq:Kpq=0}
K_{p,q}(a_1,\cdots,a_n)=0\rm{ for }0\leq p<P(a_1,\cdots,a_n;q)-q,
\end{equation}
where $P(a_1,\cdots,a_n;q)$ is defined recursively by the following rules:
\begin{enumerate}
 \item $P(a;0)=0$ and $P(a;q)=\infty$ for $q>0$.
 \item For $n>1$,
\[ P(a_1,\cdots,a_n;q)=\min_{0\leq j\leq\min(q,a_n)}(P(a_1,\cdots,a_{n-1};q-j)+j)\cdot(j+1).\]
\end{enumerate}
If $a_1=\cdots=a_n=1$ then the (non-)vanishing behavior of $K_{p,q}$ is completely characterized by
 \begin{equation}\label{eq:Kpq-allP1}
   K_{p,q}(1,\cdots,1)\neq 0\rm{ if and only if }2^{q+1}-2-q\leq p\leq 2^n-2^{n-q}-q,\rm{ and }q=0,\cdots,n-1.
 \end{equation}
\end{main*}

Note that $K_{p,q}(a_1,\cdots,a_n)=K_{p,q}(a_1,\cdots,a_n,0)$ so there is no harm in assuming that all $a_i>0$. We chose to allow some $a_i=0$ for more flexibility. To illustrate our theorem with an example, consider the Segre embedding of $\bb{P}^2 \times \bb{P}^2 \times \bb{P}^1$ into~$\bb{P}^{17}$, which occurs by taking $n=3$ and $(a_1,a_2,a_3) = (2,2,1)$. The graded Betti table (as computed by Macaulay2 \cite{M2}) records in column $p$ and row $q$ the vector space dimension of $K_{p,q}(2,2,1)$:

\bigskip

\begin{verbbox}
       0  1   2    3    4    5    6    7    8    9  10 11 12
total: 1 63 394 1179 2087 2692 3726 4383 3275 1530 407 45  2
    0: 1  .   .    .    .    .    .    .    .    .   .  .  .
    1: . 63 394 1179 1980 1702  396   63    8    .   .  .  .
    2: .  .   .    .  107  990 3330 4320 3267 1530 407 36  .
    3: .  .   .    .    .    .    .    .    .    .   .  9  2
\end{verbbox}

\begin{center}
\theverbbox
\end{center}

\bigskip

The relevant values of the $P$-function are computed below
\begin{center}
\setlength{\extrarowheight}{2pt}
\tabulinesep=1.2mm
\begin{tabu}{c|ccccc}
$q$ & $0$ & $1$ & $2$ & $3$ & $4$ \\
\hline
$P(2,2,1;q)$ & $0$ & $2$ & $6$ & $14$ & $\infty$\\
\hline
$P(2,2,1;q)-q$ & $0$ & $1$ & $4$ & $11$ & $\infty$ \\
\end{tabu}
\end{center}
showing that the vanishing statement (\ref{eq:Kpq=0}) is sharp in this example. Notice also that $K_{p,1}(2,2,1) \neq 0$ precisely for $1\leq p\leq 8 = (a_1+1)(a_2+1)+a_3-2$, so (\ref{eq:Kp1-not-0}) is also sharp.

The (non-)vanishing properties of the groups $K_{p,2}(a_1,\cdots,a_n)$ have been previously investigated by Rubei in relationship with the Green--Lazarsfeld property $N_p$: she shows in \cites{rubei-P1,rubei-PN} that for any $n\geq 3$ and any $a_1,\cdots,a_n\geq 1$ one has $K_{p,2} = 0$ for $p\leq 3$ and $K_{4,2}\neq 0$. To see how our results recover Rubei's, we first note that the non-vanishing statement for arbitrary products reduces using (\ref{eq:kpq-specialization}) to verifying that $K_{4,2}(1,1,1) \neq 0$, which in turn follows from (\ref{eq:Kpq-allP1}) by taking $q=2$. For the vanishing of the groups $K_{p,2}$ we can apply (\ref{eq:Kpq=0}) once we compute the appropriate values for the $P$-function: we invite the reader to check that 
\begin{itemize}
\item $P(a_1,\cdots,a_n;1) = 2$ if $n\geq 2$ and $a_2\geq 1$, and that $P(a,0,\cdots;1) = \infty$;
\item $P(a_1,\cdots,a_n;2) = 6$ if $n\geq 2$ and $a_2\geq 2$, or if $n\geq 3$ and $a_3\geq 1$, while in all other cases $P(a_1,\cdots,a_n;2) = \infty$.
\end{itemize}
It follows from (\ref{eq:Kpq=0}) that $K_{p,2} = 0$ for $n\geq 3$ and $a_1,\cdots,a_n\geq 1$, recovering the vanishing aspect of Rubei's result.

The vanishing statement (\ref{eq:Kpq=0}) may seem hard to apply due to the recursive nature of the $P$-function. Nevertheless, in the case when $a_i=a$ for all $i$ we get an explicit closed formula as follows. If we let $q,r\geq 0$ such that $(r-1)a<q\leq ra$ and set $q_0=q-(r-1)a$, then we show in Lemma~\ref{lem:Paaa} that
\[
P(a,a,\cdots,a;q)=\begin{cases}
(q_0^2+q_0)\cdot(a+1)^{r-1}+(a+1)^r-(a+1) & \rm{for }n>r, \\
\infty & \rm{otherwise}.
\end{cases}
\]

The paper is organized as follows. In Section~\ref{sec:prelim} we recall some basic notation and results regarding syzygies, representation theory, and the Kempf--Weyman geometric technique. In Section~\ref{sec:vanishing} we prove the vanishing statement (\ref{eq:Kpq=0}) from our Main Theorem. In Section~\ref{sec:non-vanishing-Kp1} we use Artinian reduction as in \cite{EEL} to prove the non-vanishing (\ref{eq:Kp1-not-0}) for $K_{p,1}$ groups. In Section~\ref{sec:straightening} we construct a standard basis for the Artinian reduction of the coordinate ring of the Segre variety, inspired by the theory of algebras with straightening laws, and devise an algorithm for expressing every element in the said reduction as a linear combination of elements of the standard basis. Equipped with a standard basis, we adapt the methods of \cite{EEL} to prove in Section~\ref{sec:non-vanishing-P1} a sharp non-vanishing result for the syzygies of Segre embeddings of products of $\bb{P}^1$. We conclude with a number of examples in Section~\ref{sec:examples}.

\section{Preliminaries}\label{sec:prelim}

\subsection{Syzygies}\label{subsec:syzygies}

Let $S$ denote a standard graded polynomial ring over $\bb{C}$ and let $M$ be a finitely generated graded $S$-module. For $p\geq 0$ and $q\in\bb{Z}$ we let
\begin{equation}\label{eq:Kpq-M}
 K_{p,q}^S(M) = \Tor_p^S(M,\bb{C})_{p+q},
\end{equation}
and refer to the groups $K_{p,q}^S(M)$ as the \defi{syzygy modules of $M$}. We write $K_{p,q}(M)$ instead of $K_{p,q}^S(M)$ when there is no danger of confusion regarding the ring $S$. 

It will be useful to interpret the syzygy modules with \defi{Koszul homology groups} as follows. Write $S = \Sym_{\bb{C}}(V)$ where $V$ is the vector space of linear forms (say $\dim(V)=r$), and consider the Koszul complex resolving the residue field:
\[K_{\bullet}:\qquad 0\lra \bw^r V \oo S(-r) \lra \cdots \lra \bw^p V \oo S(-p) \lra \cdots \lra S \lra 0.\]
Since $\Tor_p^S(M,\bb{C})$ is the $p$-th homology group of $K_{\bullet} \oo_S M$, it follows that we can compute $K_{p,q}(M)$ as the homology of the $3$-term complex
\begin{equation}\label{eq:3term-koszul}
 \bw^{p+1}V \oo M_{q-1} \lra \bw^p V \oo M_q \lra \bw^{p-1}V \oo M_{q+1}.
\end{equation}

Finally, suppose that $\ell_1,\cdots,\ell_k$ is a regular sequence on $M$ consisting of linear forms in $S$ and let
\[
  \ol{S} = S / (\ell_1,\cdots,\ell_k) \qquad\mbox{ and }\qquad \ol{M} = M \oo_S \ol{S}.
\]
Note that $\ol{S}$ can be thought of as a polynomial ring in $k$ fewer variables and $\ol{M}$ is a finitely generated graded $\ol{S}$-module. For each $p\geq 0$, $q\in\bb{Z}$ we have an isomorphism of vector spaces
\begin{equation}\label{eq:reduction-syz}
 K_{p,q}^S(M) \simeq K_{p,q}^{\ol{S}}(\ol{M}).
\end{equation}

\subsection{The geometric technique}\label{subsec:geometric-technique}
Here we introduce essential ingredients from the Kempf--Weyman \emph{geometric technique}, \cite[Chapter 5]{weyman}.
Let $V$ be a finite dimensional complex vector space, let $\bb{P}V$ denote the projective space of rank one quotients of $V$, and consider the tautological sequence
\begin{equation}\label{eq:taut-seq}
 0\lra\mc{R}\lra V\oo\mc{O}_{\bb{P}V}\lra\mc{Q}\lra 0
\end{equation}
where $\mc{Q} \cong \mc{O}(1)$ is the tautological rank one quotient bundle, and $\mc{R}$ is the tautological subbundle.

Suppose that $U$ is another finite dimensional vector space, and let 
\begin{equation}\label{eq:S-vs-mcS}
S = \Sym_{\bb{C}}(U\oo V)\qquad\mbox{ and }\qquad\mc{S} = \Sym_{\mc{O}_{\bb{P}V}}(U\oo\mc{Q}).
\end{equation}
We think of $S$ as a graded ring with $U\oo V$ sitting in degree one, and similarly think of $\mc{S}$ as a sheaf of graded algebras. Note that \eqref{eq:taut-seq} induces a natural degree preserving surjective map $S\oo\mc{O}_{\bb{P}V} \lra \mc{S}$.

Let $\mc{M}$ be a quasi-coherent sheaf of graded $\mc{S}$-modules on $\bb{P}V$ and assume that $\mc{M}$ has no higher cohomology. Set $M = H^0(\bb{P}V,\mc{M})$, which is a graded $S$-module in the natural way.

Suppose further that $\mc{M}$ has a \defi{minimal graded free resolution}
\begin{equation}\label{eq:def-F}
\mc{F}_{\bullet}: 0\lra\mc{F}_r\lra\cdots\lra\mc{F}_0 \lra \mc{M}\lra 0,
\end{equation}
where $\mc{F}_p = \bigoplus_{q\in\bb{Z}} \mc{K}_{p,q}(\mc{M}) \oo_{\mc{O}_{\bb{P}V}} \mc{S}$ and moreover the following conditions are satisfied:
\begin{itemize}
 \item Each $\mc{K}_{p,q}(\mc{M})$ is a coherent sheaf on $\bb{P}V$.
 \item For each $p$ there are only finitely many values of $q$ for which $\mc{K}_{p,q}(\mc{M}) \neq 0$.
 \item For $q<q'$ the induced map $\mc{K}_{p+1,q}(\mc{M}) \oo \mc{S} \lra \mc{K}_{p,q'}(\mc{M}) \oo \mc{S}$ is identically zero.
\end{itemize}

\begin{theorem}\label{thm:pushfwd-resolution}
 The module $M = H^0(\bb{P}V,\mc{M})$ is a finitely generated $S$-module and for each $p\geq 0$, $q\in\bb{Z}$ we have that the vector space $K_{p,q}(M)$ is a subquotient of
\[
\bigoplus_{i,j\geq 0}H^j\left(\bb{P}V,\bw^i(U\oo\mc{R})\oo\mc{K}_{p-i+j,q-j}(\mc{M})\right).
\]
\end{theorem}

\begin{proof}
We have a Koszul complex (we write $S$ in place of $S \otimes \mc{O}_{\bb{P}V}$):
\[
\bb{K} : \cdots \to \bw^i(U \otimes \mc{R})\otimes S \to \cdots \to \bw^2(U \otimes \mc{R}) \otimes S \to (U \otimes \mc{R}) \otimes S \to S \to \mc{S} \to 0
\]
which is acyclic since $\mc{S}$ is locally cut out by linear equations in $S \times \bb{P}V$. Since the terms are all locally free, tensoring $\bb{K}$ with $\mc{K}_{p,q}(\mc{M})$ is an exact operation, so we can replace $\mc{F}_i$ in \eqref{eq:def-F} with $\bb{K} \otimes \mc{F}_i$ to get a double complex whose total complex $\mc{T}$ has homology equal to $\mc{M}$ in degree $0$ and is $0$ in all other degrees. We will label $\mc{T}$ cohomologically, so that
\[
\mc{T}^{-i} = \bigoplus_{j \ge 0} \mc{F}_{i-j} \otimes \bigwedge^j(U \otimes \mc{R}) \otimes S.
\]
Now replace each term in $\mc{T}$ with an injective resolution and take sections to get another double complex whose total complex $\bb{T}$ has terms
\[
\bb{T}^i = \bigoplus_{i,j\geq 0} H^j(\bb{P}V, \mc{T}^i)
\]
such that $H^j(\bb{T}) = H^j(\bb{P}V, \mc{M})$. By assumption, this is $0$ if $j>0$, so we have a long exact sequence of graded $S$-modules
\[
\cdots \to \bb{T}^{-i} \to \bb{T}^{-i+1} \to \cdots \to \bb{T}^{-1} \to \bb{T}^0 \to M \to 0.
\]
In particular, each $\bb{T}^i$ is a free $S$-module, so this is a graded free resolution. Tensoring this with the residue field of $S$, we get a complex whose $p$th homological term in degree $p+q$ is 
\[
\bigoplus_{i,j\geq 0}H^j\left(\bb{P}V,\bw^i(U\oo\mc{R})\oo\mc{K}_{p-i+j,q-j}(M)\right).
\]
Furthermore, this complex computes $\Tor^S_p(M, \bb{C})_{p+q} \cong K_{p,q}(M)$, so we get the desired statement. Finally, $M$ is a finitely generated $S$-module since the assumptions above guarantee that $\Tor^S_0(M,\bb{C})_q$ is nonzero only for finitely many $q$, and each nonzero value is a finite-dimensional vector space.
\end{proof}

\subsection{Representation theory}

Let $m = \dim V$. The irreducible algebraic representations of $\GL(V)$ are given by Schur modules $S_\lambda(V)$ tensored with some power of the determinant character $\det (V)$. Here $\lambda = (\lambda_1, \dots, \lambda_m)$ is a weakly decreasing sequence of non-negative integers, i.e., a partition. We can extend the notation for Schur functors for all weakly decreasing sequences of integers by setting $S_\lambda(V) = S_\mu(V) \otimes \det(V)^{\otimes d}$ where $\mu_i = \lambda_i - d$. The irreducible polynomial representations are exactly given by $S_\lambda(V)$ where $\lambda$ is a partition. These will be the only representations we will be concerned with, though when taking duals, it is convenient to use non-polynomial representations. See \cite[Chapter 2]{weyman} for background, but note that our indexing conventions coincide with the Weyl functors defined there, not the Schur functors.

Let $\Sigma_m$ denote the group of bijections of $[m] = \{1,\dots,m\}$. Given $\sigma \in \Sigma_m$, define its length to be
\[
\ell(\sigma) = \#\{(i,j) \mid 1 \le i < j \le m, \quad \sigma(i) > \sigma(j)\}.
\]
Also define 
\[
\rho = (m-1, m-2, \dots, 1 ,0) \in \bb{Z}^m.
\]
Given $v \in \bb{Z}^m$, define $\sigma(v) = (v_{\sigma^{-1}(1)}, \dots, v_{\sigma^{-1}(m)})$ and $\sigma \bullet v = \sigma(v + \rho) - \rho$. Note that given any $v \in \bb{Z}^m$, either there exists $\sigma \ne 1$ such that $\sigma \bullet v = v$, or there exists a unique $\sigma$ such that $\sigma \bullet v$ is weakly decreasing.

\begin{theorem}[Borel--Weil--Bott]
Let $\alpha = (\alpha_1, \dots, \alpha_{m-1}) \in \bb{Z}^{m-1}$ be a weakly decreasing sequence and set $v = (d, \alpha_1, \dots, \alpha_{m-1})$ for $d \in \bb{Z}$. Exactly one of the following two cases happens:
\begin{enumerate}
\item There exists $\sigma \ne 1$ such that $\sigma \bullet v = v$. In that case, $H^j(\bb{P}V; \mc{Q}^d \otimes S_\alpha(\mc{R})) = 0$ for all $j$.
\item Otherwise, there exists unique $\sigma$ such that $\beta = \sigma \bullet v$ is weakly decreasing, in which case
\[
H^j(\bb{P}V; \mc{Q}^d \otimes S_\alpha(\mc{R})) = 
\begin{cases} 
S_\beta(V) & \text{if $j = \ell(\sigma)$}\\
0 & \text{if $j \ne \ell(\sigma)$} \end{cases}.
\]
The equality denotes that the two terms are isomorphic as $\GL(V)$-representations.
\end{enumerate}
\end{theorem}

See \cite[Corollary 4.1.9]{weyman}. 

\section{Vanishing of syzygies of Segre products}\label{sec:vanishing}

Consider a sequence of non-negative integers $\ul{a} = (a_1\geq\cdots\geq a_n)$ and suppose that $\ul{V} = (V_1,\cdots,V_n)$ is a tuple of vector spaces with $\dim(V_i) = a_i + 1$ for $i=1,\cdots,n$. The \defi{Segre variety $X=\Seg(\ul{V})$} is the image of the embedding
\[\bb{P}V_1 \times \cdots \times \bb{P}V_n \lra \bb{P}(V_1 \oo \cdots \oo V_n),\quad ([v_1],\cdots,[v_n]) \mapsto [v_1 \oo \cdots \oo v_n].\]
We let $S = \Sym_{\bb{C}}(V_1 \oo \cdots \oo V_n)$ be the homogeneous coordinate ring of $\bb{P}(V_1 \oo \cdots \oo V_n)$ and let
\[R = \bigoplus_{d\geq 0} \Sym^d V_1 \oo \cdots \oo \Sym^d V_n\]
be the homogeneous coordinate ring of $\Seg(\ul{V})$. We will be interested in studying the syzygy modules
\[K_{p,q}(\ul{V}) = \Tor_p^S(R,\bb{C})_{p+q}.\]
The description (\ref{eq:3term-koszul}) of the $K_{p,q}$ groups, where $V = V_1 \oo \cdots \oo V_n$ and $M_q = \Sym^q V_1 \oo \cdots \oo \Sym^q V_n$, shows that the formation of $K_{p,q}(\ul{V})$ is functorial in $V_1,\cdots,V_n$, and in fact it is given by a polynomial functor. In particular, each $K_{p,q}(\ul{V})$ is a representation of the group $\GL = \GL(V_1) \times \cdots \times \GL(V_n)$. The functoriality shows that the vanishing behavior of $K_{p,q}(\ul{V})$ only depends on the dimension vector $\ul{a}$ and not on the vector spaces $V_i$ themselves, so we will often write $K_{p,q}(\ul{a})$ instead of $K_{p,q}(\ul{V})$, with no danger of confusion. The fact that each $K_{p,q}(\ul{V})$ is a polynomial functor shows that
\begin{equation}\label{eq:kpq-specialization}
\mbox{if }a_i \geq b_i \mbox{ for }i=1,\cdots,n\mbox{ and if }K_{p,q}(\ul{a}) = 0\mbox{ then }K_{p,q}(\ul{b}) = 0.
\end{equation}

We now describe the vanishing result. Consider the function $h \colon \bb{Z}\times\bb{Z} \lra \bb{Z}$ defined by
\[
h(x,j)=(x+j)\cdot(j+1),
\]
and define $P(a_1,\cdots,a_n;q)$ recursively for $a_1\geq\cdots\geq a_n\geq 0$ and $q\geq 0$ by the following rules:
\begin{enumerate}
 \item $P(a;0)=0$ and $P(a;q)=\infty$ for $q>0$.
 \item For $n>1$,
 \begin{equation}\label{eq:recursion}
 P(a_1,\cdots,a_n;q)=\min_{0\leq j\leq\min(q,a_n)}h(P(a_1,\cdots,a_{n-1};q-j),j).
 \end{equation}
\end{enumerate}

\begin{theorem}\label{thm:vanishing}
 For $q\geq 0$ and $a_1\geq\cdots\geq a_n\geq 0$, we have
\[K_{p,q}(a_1,\cdots,a_n)=0\rm{ for }0\leq p<P(a_1,\cdots,a_n;q)-q.\]
\end{theorem}

\begin{remark}
 It is not hard to prove by induction on $n$ that 
 \[P(a_1,\cdots,a_n;q)=\infty\mbox{ for }q>a_2+\cdots+a_n,\]
 which means that for such values of $q$ we have $K_{p,q}(\ul{a})=0$ for all $p\geq 0$. This is a reflection of the fact that the Castelnuovo--Mumford regularity of $R$ is precisely equal to $a_2+\cdots+a_n$, so for the purpose of detecting regularity our result is sharp.
\end{remark}

\begin{remark}\label{rem:2factor-Segre}
 In the case when $n=2$ one gets that $P(a_1,a_2;q)=q^2+q$ for $q\leq a_2$, so $K_{p,q}(a_1,a_2)=0$ whenever $p<q^2$. In this case our result is again sharp, detecting the precise vanishing for the syzygies of $2$-factor Segre products (see \cite[Theorem~6.1.4]{weyman} for the explicit description of all the $K_{p,q}$ groups).
\end{remark}

The behavior of the function $P(a_1,\cdots,a_n;q)$ is hard to describe in general, so we specialize to the case when $a_1=\cdots=a_n=a$ when a more concrete formula can be given. In this case we have

\begin{lemma}\label{lem:Paaa}
 Consider non-negative integers $a,q,r$, such that $(r-1)a<q\leq ra$. If we let $q_0=q-(r-1)a$, and write $a^n$ for the sequence $(a_1,\cdots,a_n)$ with $a_i=a$, then we have
\[P(a^n;q)=\begin{cases}
(q_0^2+q_0)\cdot(a+1)^{r-1}+(a+1)^r-(a+1) & \rm{for }n>r, \\
\infty & \rm{otherwise}.
\end{cases}
\]
\end{lemma}

\begin{remark}\label{rem:a-divides-q}
 If we take $r=n-1$ and $q=ra$ then we get $P(a^n;(n-1)a)=(a+1)^n-(a+1)$. Theorem~\ref{thm:vanishing} then predicts that
\[K_{p,(n-1)a}(a^n)=0\rm{ for }p<(a+1)^n-1-na.\]
Since $R$ is Gorenstein of projective dimension $(a+1)^n-1-na$, and since its Castelnuovo--Mumford regularity is $(n-1)a$, it follows that $K_{p,(n-1)a}(a^n) \neq 0$ if and only if $p = (a+1)^n-1-na$, so our vanishing result is again sharp in this case. In fact this argument shows that for every $0\leq r < n$ we have that
\[K_{(a+1)^{r+1}-1-(r+1)a,ra}(a^{r+1}) \neq 0\]
which by (\ref{eq:kpq-specialization}) implies that 
\[K_{(a+1)^{r+1}-1-(r+1)a,ra}(a^n) \neq 0,\]
and since $P(a^n;ra) - ra = (a+1)^{r+1}-1-(r+1)a$, the vanishing result is sharp whenever $a$ divides $q$.
\end{remark}

\begin{proof}[Proof of Lemma~\ref{lem:Paaa}]
 We prove the formula for $P(a^n;q)$ by induction on $n$: when $n=1$, it asserts that $P(a;0) = 0$ and $P(a;q) = \infty$ for $q>0$, which is just condition (1) in the definition of the function $P$. If $n=2$ then our formula asserts that $P(a^2;q) = q^2+q$ for $0\leq q\leq a$ and $P(a^2;q) = \infty$ for $q>a$, which is a special case of Remark~\ref{rem:2factor-Segre}. If $q=0$ then $r=1$, $q_0=a$ and our formula reads
 \[P(a^n;0) = (a^2+a)\cdot(a+1)^{-1} + (a+1)^0 - (a+1) = 0,\]
 which follows from (\ref{eq:recursion}) since $P(a^n;0) = P(a^{n-1};0)$.
 
 We assume from now on that $n\geq 3$ and $q\geq 1$. If $q>(n-1)a$ then $q-j > (n-2)a$ for every $0\leq j\leq\min(q,a)$ and therefore $P(a^{n-1},q-j) = \infty$ by induction. It follows from (\ref{eq:recursion}) that $P(a^n;q) = \infty$ as well. We may then assume that $q\leq(n-1)a$, so we can find $1\leq r<n$ satisfying $(r-1)a<q\leq ra$, and we can set $q_0 = q - (r-1)a$.
 
 We begin by proving that $P(a^n;q) \leq (q_0^2+q_0)\cdot(a+1)^{r-1}+(a+1)^r-(a+1)$, for which according to (\ref{eq:recursion}) it is sufficient to find $0\leq j\leq\min(q,a)$ such that
 \begin{equation}\label{eq:upperbound-Panq}
 (P(a^{n-1};q-j) + j ) \cdot (j+1) = (q_0^2+q_0)\cdot(a+1)^{r-1}+(a+1)^r-(a+1)
 \end{equation}
 If $r\leq n-2$ then it follows by induction that $P(a^{n-1};q) = (q_0^2+q_0)\cdot(a+1)^{r-1}+(a+1)^r-(a+1)$ so we can take $j=0$ in (\ref{eq:upperbound-Panq}). If $r=n-1\geq 2$ then we get by induction that 
 \[P(a^{n-1};q-a) = (q_0^2+q_0)\cdot(a+1)^{r-2}+(a+1)^{r-1}-(a+1)\]
and since $q > (r-1)a \geq a$, we can take $j=a$ in (\ref{eq:upperbound-Panq}).
 
To conclude our induction we need to verify that
  \begin{equation}\label{eq:lowerbound-Panq}
 (P(a^{n-1};q-j) + j ) \cdot (j+1) \geq (q_0^2+q_0)\cdot(a+1)^{r-1}+(a+1)^r-(a+1) \mbox{ for all }0\leq j\leq\min(q,a).
 \end{equation}
We let $q' = q-j$ and consider two possible scenarios according to whether $q'>(r-1)a$ or $q'\leq(r-1)a$.

\ul{$(r-1)a < q' \leq ra$}: we set $q'_0 = q' - (r-1)a = q_0 - j$. If $r=n-1$ then $P(a^{n-1};q')=\infty$ so (\ref{eq:lowerbound-Panq}) is trivial. We may thus assume that $r\leq n-2$ and therefore
\[P(a^{n-1};q') = ((q'_0)^2+q'_0)\cdot(a+1)^{r-1}+(a+1)^r-(a+1).\]
Using the fact that $q_0^2+q_0 - ((q'_0)^2+q'_0) = j\cdot(q_0 + q'_0 + 1)$ we can rewrite (\ref{eq:lowerbound-Panq}) as
\[P(a^{n-1};q')\cdot j + j\cdot(j+1) \geq j\cdot(q_0 + q'_0 + 1) \cdot (a+1)^{r-1}.\]
If $j=0$ this is an equality, otherwise we can divide by $j$ and rewrite it as
\[((q'_0)^2+q'_0+a+1)\cdot(a+1)^{r-1} + j-a \geq (q_0 + q'_0 + 1) \cdot (a+1)^{r-1}\]
which is equivalent to 
\[((q'_0)^2+a-q_0)\cdot(a+1)^{r-1} \geq a-j.\]
Since $q'_0\geq 1$ we get $(q'_0)^2 + a - q_0 \geq q'_0 + a - q_0 = a-j$, while since $r\geq 1$ we get $(a+1)^{r-1}\geq 1$, from which the inequality follows.

\ul{$(r-2)a < q' \leq (r-1)a$}: we set $q'_0 = q' - (r-2)a = q_0 + a - j$. If $r=1$ then $q'=0$, $q_0=q=j$, and (\ref{eq:lowerbound-Panq}) is an equality. We may thus assume that $r\geq 2$, in which case
\[P(a^{n-1};q') = ((q'_0)^2+q'_0)\cdot(a+1)^{r-2}+(a+1)^{r-1}-(a+1).\]
Using the fact that $(q'_0)^2+q'_0 - (q_0^2 + q_0) = (a-j)\cdot(q'_0 + q_0 + 1)$, and adding $P(a^{n-1};q')\cdot(a-j) - j(j+1) + (a+1)^2$ to both sides of (\ref{eq:lowerbound-Panq}) we can rewrite the inequality as
\[(a-j)\cdot(q'_0 + q_0 + 1)\cdot(a+1)^{r-1} \geq (a-j)\cdot (P(a^{n-1};q') + j + a + 1).\]
If $a=j$ then this is an equality, otherwise we can divide by $a-j$ and rewrite it as
\[(q'_0 + q_0)\cdot(a+1)^{r-1} \geq ((q'_0)^2+q'_0)\cdot(a+1)^{r-2} + j,\]
which is equivalent to
\[(q'_0\cdot a + q_0\cdot a + q_0 - (q'_0)^2)\cdot(a+1)^{r-2}\geq j.\]
Since $r\geq 2$ we get $(a+1)^{r-2}\geq 1$. Since $q'_0\leq a$ we have $q'_0\cdot a - (q'_0)^2 \geq 0$. Since $q_0>0$, we get $q_0\cdot a + q_0 \geq a+1 > j$, so the desired inequality follows, concluding the proof.
\end{proof}

\begin{proof}[Proof of Theorem~\ref{thm:vanishing}]
We prove the vanishing result by induction on $n$, the case $n=1$ being trivial: $R=S$ and $K_{p,q}(a_1)=0$ unless $p=q=0$.

Assume from now on that $n>1$, and consider the tautological sequence on $\bb{P}V_n$:
\[0\lra\mc{R}\lra V_n\oo\mc{O}_{\bb{P}V_n}\lra\mc{Q}\lra 0,\]
where $\mc{Q}$ is the tautological rank one quotient bundle, and $\mc{R}$ is the tautological subbundle. We use the notation (\ref{eq:S-vs-mcS}), with $U = V_1 \oo \cdots \oo V_{n-1}$ and $V = V_n$. We consider the quasi-coherent sheaf of $\mc{S}$-algebras
\[\mc{M} = \bigoplus_{d\geq 0} \Sym^d V_1 \oo \cdots \oo \Sym^d V_{n-1} \oo \mc{Q}^d,\]
and note that $R = H^0(\bb{P}V_n,\mc{M})$. By functoriality, $\mc{M}$ has a minimal graded free resolution (\ref{eq:def-F}) where
\[\mc{K}_{p,q}(\mc{M}) = K_{p,q}(V_1,\cdots,V_{n-1}) \oo \mc{Q}^{p+q}.\]

Theorem~\ref{thm:pushfwd-resolution} yields that $K_{p,q}(V_1,\cdots,V_n)$ is a subquotient of
\begin{subequations}
\begin{equation}\label{eq:dirsumcoh}
\bigoplus_{i,j\geq 0}H^j\left(\bb{P}V_n,\bw^i(V_1\oo\cdots\oo V_{n-1}\oo\mc{R})\oo\mc{Q}^{p+q-i}\right)\oo K_{p-i+j,q-j}(V_1,\cdots,V_{n-1}). 
\end{equation}
Suppose that the conclusion of Theorem~\ref{thm:vanishing} is false, so that $K_{p,q}(a_1,\cdots,a_n)\neq 0$ for some $p,q$ with
\begin{equation}\label{eq:p+q=small}
p+q<P(a_1,\cdots,a_n;q). 
\end{equation}
We can therefore find $i,j\geq 0$ such that the corresponding term in (\ref{eq:dirsumcoh}) is non-zero. It follows that $K_{p-i+j,q-j}(V_1,\cdots,V_{n-1})\neq 0$, so we get by induction that
\begin{equation}\label{eq:ineq1}
p+q-i\geq P(a_1,\cdots,a_{n-1};q-j). 
\end{equation}
Moreover, we must also have $H^j(\bb{P}V_n,\bw^i(V_1\oo\cdots\oo V_{n-1}\oo\mc{R})\oo\mc{Q}^{p+q-i})\neq 0$. Using Cauchy's identity \cite[Corollary 2.3.3]{weyman}, to get the decomposition
\[
\bw^i(V_1\oo\cdots\oo V_{n-1}\oo\mc{R}) = \bigoplus_{|\ll| = i} S_{\ll}\mc{R} \oo S_{\ll^\dagger}(V_1\oo\cdots\oo V_{n-1})
\]
(here $\ll^\dagger$ is the transpose partition of $\ll$), we conclude that there exists a partition $\ll$ of $i$ such that $H^j(\bb{P}V_n,S_{\ll}\mc{R}\oo\mc{Q}^{p+q-i})\neq 0$, which by Bott's theorem is possible only when $\ll_j\geq p+q-i+j+1$, which in particular implies
\[
i = |\ll| \geq \ll_1+\cdots+\ll_j \ge j\cdot \ll_j \geq j\cdot(p+q-i+j+1).
\]
For such a partition to exist it is thus necessary that
\begin{equation}\label{eq:ineq2}
i\geq j\cdot(p+q-i+j+1)=j\cdot(j+1)+j\cdot(p+q-i)\overset{(\ref{eq:ineq1})}{\geq}j\cdot(j+1)+j\cdot P(a_1,\cdots,a_{n-1};q-j).
\end{equation}
Combining (\ref{eq:ineq1}) with (\ref{eq:ineq2}) we get that
\[p+q=(p+q-i)+i\geq (P(a_1,\cdots,a_{n-1};q-j)+j)\cdot(j+1)=h(P(a_1,\cdots,a_{n-1};q-j),j),\]
which contradicts (\ref{eq:p+q=small}).
\end{subequations}
\end{proof}

\section{The non-vanishing of $K_{p,1}$ via Artinian reduction}\label{sec:non-vanishing-Kp1}

Given a sequence of positive integers $\ul{a}=(a_1\geq\cdots\geq a_n)$, we let $X=\rm{Seg}(a_1,\cdots,a_n)$ denote the image of the Segre embedding
\[\bb{P}^{a_1}\times\cdots\times\bb{P}^{a_n}\lra\bb{P}^A,\]
where $A=(a_1+1)\cdots(a_n+1)-1$. Let $S$ denote the homogeneous coordinate ring of $\bb{P}^A$, let $I\subset S$ denote the ideal of equations vanishing on $X$, and let $R=S/I$ denote the homogeneous coordinate ring of $X$. If we let $K_{p,1}(\ul{a}) = K_{p,1}^S(R)$ then the main result of this section is the following.

\begin{theorem}\label{thm:nonvanishing-Kp1}
With notation as above we have
\[K_{p,1}(\ul{a}) \neq 0 \mbox{ for } 1\leq p \leq (a_1+1)\cdots(a_{n-1}+1) + a_n - 2.\]
\end{theorem}

To prove the theorem we will construct an explicit Artinian reduction $\ol{R}$ of $R$ obtained by modding out by a regular sequence of linear forms. Inspired by \cite{EEL}, we then exhibit explicit cycles in Koszul homology in order to prove the non-vanishing of $K_{p,1}^{\ol{S}}(\ol{R})$, which by (\ref{eq:reduction-syz}) is equivalent to the non-vanishing of $K_{p,1}(\ul{a})$.

We begin by writing
\[S = \bb{C}[z_{i_1,\cdots,i_n} : 0\leq i_k\leq a_k\mbox{ for }k=1,\cdots,n]\]
in which case the ideal $I$ will be generated by binomials
\[z_{i_1,\cdots,i_n} \cdot z_{j_1,\cdots,j_n} - z_{i'_1,\cdots,i'_n} \cdot z_{j'_1,\cdots,j'_n},\mbox{ where }\{i_k,j_k\} = \{i'_k,j'_k\}\mbox{ as multisets for each }k=1,\cdots,n.\]

We next consider the poset
\begin{equation}\label{eq:def-Pa}
\mc{P}(\ul{a})=\{\ul{i}=(i_1,\cdots,i_n):0\leq i_k\leq a_k\rm{ for }k=1,\cdots,n\},
\end{equation}
with the ordering given by $\ul{i}\leq\ul{j}$ if $i_k\leq j_k$ for all $k$. We write $|\ul{i}| = i_1+\cdots+i_n$, let
\begin{align*}
\min(\ul{i},\ul{j}) &= (\min(i_1,j_1),\cdots,\min(i_n,j_n)),\\
\max(\ul{i},\ul{j}) &= (\max(i_1,j_1),\cdots,\max(i_n,j_n)).
\end{align*}
The coordinate ring $R$ of the Segre variety is an ASL (algebra with straightening law, see \cites{eisenbud-ASL,hodge} for an introduction to this subject) on the poset $\mc{P}(\ul{a})$, with the straightening law
\begin{equation}\label{eq:SL}
z_{\ul{i}}\cdot z_{\ul{j}}=z_{\min(\ul{i},\ul{j})}\cdot z_{\max(\ul{i},\ul{j})}.
\end{equation}
We define for $0\leq r\leq a_1+\cdots+a_n$ the element
\begin{equation}\label{eq:defxr}
 x_r = \sum_{|\ul{i}|=r} z_{\ul{i}}.
\end{equation}

It follows from \cite[Proposition~2.1]{bruns-ASL} that (\ref{eq:defxr}) provides a system of parameters for the maximal homogeneous ideal of $R$. We write 
\[ \ol{S} = S / (x_0,\cdots,x_{a_1+\cdots+a_n})\qquad\mbox{ and }\qquad\ol{R} = R / (x_0,\cdots,x_{a_1+\cdots+a_n}) = R\oo_S \ol{S}.\]
Since $R$ is Cohen--Macaulay, $\ol{R}$ is a finite dimensional vector space (called an \defi{Artinian reduction} of $R$) of dimension equal to the degree of the Segre variety. This is computed by the multinomial coefficient 
 \begin{align} \label{eq:deg-segre}
\deg(R) = \dim_{\bb{C}}(\ol{R}) = {a_1+\cdots+a_n \choose a_1,\cdots,a_n} = \frac{(a_1+\cdots+a_n)!}{a_1! \cdots a_n!}.
\end{align}
We will construct an explicit basis for $\ol{R}$ below, but before doing so we use the Artinian reduction $\ol{R}$ to prove the non-vanishing of the $K_{p,1}$ groups. We start with the following.

\begin{lemma}\label{lem:basis-R1}
 To any subset $\mc{V} \subseteq \mc{P}(\ul{a})$ we associate a set of linear forms in $\ol{R}_1$
 \begin{equation}\label{eq:Z-of-V}
  \mc{Z}(\mc{V}) = \{z_{\ul{v}} : \ul{v} \in \mc{V}\}.
 \end{equation}
 We have that $\mc{Z}(\mc{V})$ is a linearly independent set if and only if for each $0\leq k \leq |\ul{a}|$ there exists an element $\ul{v}^k \in \mc{P}(\ul{a}) \setminus \mc{V}$ with $|\ul{v}^k| = k$. The set $\mc{Z}(\mc{V})$ is a basis of $\ol{R}_1$ if and only if the choice for $\ul{v}^k$ is unique for each $k$.
\end{lemma}

\begin{proof}
 We have that $\{z_{\ul{u}}:\ul{u}\in\mc{P}(\ul{a})\}$ is a basis of $R_1=S_1$. It follows from (\ref{eq:defxr}) that $\ol{R}_1$ is the quotient of $R_1$ by the relations
 \begin{equation}\label{eq:rels-z}
 \sum_{\substack{\ul{v}\in \mc{P}(\ul{a}) \\ |\ul{v}| = k }} z_{\ul{v}} = 0, \qquad 0\leq k\leq |\ul{a}|.
 \end{equation}
These relations are independent since they involve disjoint sets of variables, hence the conclusion follows.
\end{proof}


\begin{proof}[Proof of Theorem~\ref{thm:nonvanishing-Kp1}]
 Consider the element $z = z_{0,0,\cdots,0,1} \in \ol{R}$. Note that by (\ref{eq:SL}) and the fact that $x_0 = z_{\ul{0}} = 0$ in $\ol{R}_1$, it follows that $z$ is annihilated by
 \[z_{i_1,\cdots,i_{n-1},0}\mbox{ where } 0\leq i_k\leq a_k\mbox{ for }k=1,\cdots,n-1,\mbox{ and }i_k\neq 0\mbox{ for at least one value of }k.\]
 We claim that $z$ is also annihilated by
 \[z_{0,\cdots,0,j}\mbox{ where }j=2,\cdots,a_n.\]
 To see this, we use (\ref{eq:rels-z}) to express
 \[ 
z_{0,0,\cdots,0,1} = - z_{1,0,\cdots,0,0} - z_{0,1,\cdots,0,0} - \cdots - z_{0,0,\cdots,1,0},
\]
 and apply (\ref{eq:SL}). It follows from Lemma~\ref{lem:basis-R1} that the set $\mc{Z}({\mc{V}})$ is linearly independent in $\ol{R}_1$ where
 \begin{align*}
\mc{V} &= \{(i_1,\cdots,i_{n-1},0) : 0\leq i_k\leq a_k\mbox{ and some }i_k \neq 0\mbox{ for }1\leq k\leq n-1\}\\
& \qquad \qquad \cup \{(0,\cdots,0,j): 2\leq j \leq a_n\}.
\end{align*}
Observe that $|\mc{V}| = (a_1+1)\cdots(a_{n-1}+1) + a_n - 2$. Using the fact that $\ol{S}_1 = \ol{R}_1$, in order to prove the non-vanishing of $K_{p,1}(\ul{a}) = K_{p,1}^{\ul{S}}(\ul{R})$ it is enough to construct a $p$-cycle in the $3$-term complex obtained by setting $V=\ol{S}_1=\ol{R}_1$ and $M=\ol{R}$ in (\ref{eq:3term-koszul}),
 \[
  \bw^{p+1}\ol{R}_1 \lra \bw^p \ol{R}_1 \oo \ol{R}_1 \lra \bw^{p-1}\ol{R}_1 \oo \ol{R}_2,
 \]
 and to show that this cycle is not a boundary. Fix $1\leq p\leq |\mc{V}|$, consider an arbitrary subset $\mc{W} \subset \mc{V}$ with $|\mc{W}| = p$, and define
 \[
   c = \left(\bw_{\ul{w} \in \mc{W}} z_{\ul{w}} \right)\oo z \in \bw^p \ol{R}_1 \oo \ol{R}_1.
 \]
 Since for all $\ul{w} \in \mc{W}$ we have $z\cdot z_{\ul{w}} = 0$ in $\ol{R}_2$, it follows that $c$ is a cycle. If $c$ was a boundary, then it would also be a cycle in the Koszul complex
  \[
  \bw^{p+1}\ol{R}_1 \lra \bw^p \ol{R}_1 \oo \ol{R}_1 \lra \bw^{p-1}\ol{R}_1 \oo \Sym^2(\ol{R}_1).
 \]
This is not the case since the products $z\cdot z_{\ul{w}} \in \Sym^2(\ol{R}_1)$ are linearly independent.
\end{proof}

\section{Standard basis and the straightening algorithm}\label{sec:straightening}

Keeping the notation from the previous section, our goal is to construct a natural basis for the graded Artinian ring $\R$, which we will call the \defi{standard basis}, and to give an algorithm for expressing every element of $\R$ as a linear combination of standard basis elements. We will refer to this algorithm as the \defi{straightening algorithm}. 

\subsection{The standard basis}

To construct an explicit basis for $\ol{R}$, we introduce the collection of \defi{increasing lattice paths} from $\ul{0}$ to $\ul{a}=(a_1,\cdots,a_n)$, defined by
\[ 
\mc{L}(\ul{a}) = \left\{ \gamma \colon \{0,1,\cdots,|\ul{a}|\} \lra \bb{Z}^n \left|
\begin{array}{l}
\gamma(0) = \ul{0},  \\
\gamma(|\ul{a}|) = \ul{a},  \\
|\gamma(t)| = t\mbox{ for all }t=0,\cdots,|\ul{a}|,  \\
\gamma(t) \leq \gamma(t+1)\mbox{ for all }t=0,\cdots,|\ul{a}|-1.
\end{array}
\right.\right\}. 
\] 

The data of a path $\gamma\in\mc{L}(\ul{a})$ is equivalent to that of a sequence $s(\gamma)_{1\leq t\leq |\ul{a}|}$ where $s(\gamma)_t$ records which coordinate jumps by one at step $t$: $s(\gamma)_t$ is equal to the unique index $i$ for which $\gamma(t)_i > \gamma(t-1)_i$. We define the \defi{descents of a path $\gamma$} (or those of the corresponding sequence $s(\gamma)$) to be the set
\begin{equation}\label{eq:desc-gamma}
\op{Desc}(\gamma) = \{t : s(\gamma)_t > s(\gamma)_{t+1}\} \subset \{1,\cdots,|\ul{a}|-1\}.
\end{equation}
For each $\gamma\in\mc{L}(\ul{a})$ we define
\begin{equation}\label{eq:m-gamma}
 m_{\gamma} = \prod_{t\in\rm{Desc}(\gamma)} z_{\gamma(t)} \in \ol{R}.
\end{equation}

\begin{example} If $\ul{a}=(1,1,1)$ and $\gamma(0) = \ul{0}$, $\gamma(1)=(0,0,1)$, $\gamma(2)=(0,1,1)$, $\gamma(3) = (1,1,1)$ then $s(\gamma) = (3,2,1)$, $\op{Desc}(\gamma) = \{1,2\}$ and
\[m_{\gamma} = z_{0,0,1} \cdot z_{0,1,1}.\]
On the other hand, if $\gamma(0) = \ul{0}$, $\gamma(1)=(0,0,1)$, $\gamma(2)=(1,0,1)$, $\gamma(3) = (1,1,1)$ then $s(\gamma) = (3,1,2)$, $\op{Desc}(\gamma) = \{1\}$ and
\[
m_{\gamma} = z_{0,0,1}. \qedhere
\]
\end{example}

\begin{theorem}\label{thm:basis-Rbar}
 The set 
 \begin{equation}\label{eq:standard-basis}
 \mc{B}(\ul{a}) = \{m_{\gamma}:\gamma\in\mc{L}(\ul{a})\}
 \end{equation}
 is a basis of $\ol{R}$, called the \defi{standard basis}.
\end{theorem}

\begin{remark}
 Note that the maximal degree of a monomial $m_{\gamma}$ for $\gamma\in\mc{L}(\ul{a})$ is $a_2+\cdots+a_n$ (there exist $a_1$ values of $t$ for which $s(\gamma)_t = 1$, and none of these values belongs to $\op{Desc}(\gamma)$). This is consistent with the fact that $a_2+\cdots+a_n$ is the Castelnuovo--Mumford regularity of $R$ as an $S$-module (or equivalently that of $\ol{R}$ as an $\ol{S}$-module).
\end{remark}

In order to prove Theorem~\ref{thm:basis-Rbar} we fix some notation and establish a series of preliminary results. Given a monomial
\begin{equation}\label{eq:some-m}
m = z_{\ul{v}^1} \cdots z_{\ul{v}^r}\mbox{ where }|\ul{v}^1| \leq \cdots \leq |\ul{v}^r|\mbox{ and }\ul{v}^i\in\mc{P}(\ul{a})
\end{equation}
we define the \defi{lexicographic rank} of $m$, denoted $\lexrk(m)$, to be 
\[\lexrk(m) = (|\ul{v}^1|,\cdots,|\ul{v}^r|) \in \bb{Z}^r_{\geq 0}.\]
We consider the lexicographic order on $\bb{Z}^r_{\geq 0}$ given by
\[
(d_1,\cdots,d_r) \prec (e_1,\cdots,e_r)\mbox{ if there exists an index }i \le r\mbox{ with }d_j = e_j,\mbox{ for }j < i,\mbox{ and }d_{i}<e_{i},
\]
and the induced order on monomials: $m\prec m'$ if $\lexrk(m)\prec\lexrk(m')$. This is not to be confused with the partial order defining the poset $\mc{P}(\ul{a})$ in (\ref{eq:def-Pa}), which we write as $<$. An important fact which we will use throughout is that the lexicographic order is compatible with multiplication: if $m\prec m'$ are monomials of degree $r$, and if $m''$ is any other monomial then $m\cdot m'' \prec m'\cdot m''$.

We define the \defi{rank} of a monomial $m$ as in (\ref{eq:some-m}) to be
\[\rk(m) = |\lexrk(m)| = |\ul{v}^1| + \cdots + |\ul{v}^r| \in \bb{Z}_{\geq 0}.\]
Notice that the notion of rank induces a $\bb{Z}$-grading on $S$, and that the relations (\ref{eq:SL}) and (\ref{eq:rels-z}) that are used to define $\R$ as a quotient of $S$ are homogeneous with respect to this grading. It follows that $\R$ inherits a $\bb{Z}^2$-grading given by degree and rank, respectively.

\begin{lemma}\label{lem:uv-to-ww'}
 If $\ul{u},\ul{v}\in\mc{P}(\ul{a})$ and $|\ul{u}| = |\ul{v}|$ then we have an expression in $\ol{R}$
\[
z_{\ul{u}} \cdot z_{\ul{v}} = \sum c_{\ul{w},\ul{w}'}\cdot z_{\ul{w}} \cdot z_{\ul{w}'}
\]
where $c_{\ul{w},\ul{w}'} \in \bb{Z}$ and each term in the sum has $|\ul{w}| < |\ul{u}|$. In other words, $z_{\ul{u}} \cdot z_{\ul{v}}$ can be expressed as a linear combination of monomials of smaller lexicographic rank.
\end{lemma}

\begin{proof} If $\ul{u} \not\leq \ul{v}$ then we use (\ref{eq:SL}) to get
\[ z_{\ul{u}} \cdot z_{\ul{v}} = z_{\min(\ul{u},\ul{v})}\cdot z_{\max(\ul{u},\ul{v})},\]
so the conclusion follows by taking $w=\min(\ul{u},\ul{v})$ and $w'=\max(\ul{u},\ul{v})$ and $c_{\ul{w},\ul{w}'}=1$. We may thus assume that $\ul{u} \leq \ul{v}$, which together with the fact that $|\ul{u}| = |\ul{v}|$ implies $\ul{u} = \ul{v}$. Using (\ref{eq:rels-z}) we can write
\begin{equation}\label{eq:re-express-zv}
z_{\ul{v}} = - \sum_{\substack{\ul{v} \neq \ul{v}' \in \mc{P}(\ul{a}) \\ |\ul{v}'| = |\ul{v}| }} z_{\ul{v}'},
\end{equation}
and the conclusion follows from the previous case by observing that all the terms in the above sum satisfy $|\ul{v}'| = |\ul{u}|$ and $\ul{v}'\neq\ul{u}$.
\end{proof}

\begin{corollary}\label{cor:incr-ui}
 The degree $r$ homogeneous component $\ol{R}_r$ of $\ol{R}$ is spanned by monomials
 \begin{equation}\label{eq:incr-ui}
 z_{\ul{u}^1} \cdots z_{\ul{u}^r}\mbox{ with }\ul{u}^i\in\mc{P}(\ul{a})\mbox{ and }\ul{u}^1 < \cdots < \ul{u}^r.
 \end{equation}
\end{corollary}

\begin{proof}
 The vector space $\ol{R}_r$ is spanned by monomials $z_{\ul{v}^1} \cdots z_{\ul{v}^r}$ where $|\ul{v}^1| \leq \cdots \leq |\ul{v}^r|$ and $\ul{v}^i\in\mc{P}(\ul{a})$. We assume that the conclusion of the corollary is false, and consider a minimal monomial $m=z_{\ul{v}^1} \cdots z_{\ul{v}^r}$ which is not expressible as a linear combination of monomials as in (\ref{eq:incr-ui}). We let $1\leq i< r$ be the smallest index for which $\ul{v}^i \nless \ul{v}^{i+1}$. If $\ul{v}^i = \ul{v}^{i+1}$ then we apply Lemma~\ref{lem:uv-to-ww'}, while if $\ul{v}^i \neq \ul{v}^{i+1}$ we apply the relation (\ref{eq:SL}) to re-express $m$ as a linear combination of monomials $m_j$ of smaller lexicographic rank. By the minimality of $m$, all the monomials $m_j$ are expressible as linear combinations of monomials as in (\ref{eq:incr-ui}). This means that the same must be true about $m$, contradicting our assumption.
\end{proof}

For any $\ul{u}<\ul{v}\in \mc{P}(\ul{a})$ we let
\begin{align*}
f(\ul{u},\ul{v}) &= \min\{ i : u_i < v_i\},\mbox{ the first index where }\ul{u}\mbox{ and }\ul{v}\mbox{ differ, and} \\
l(\ul{u},\ul{v}) &= \max\{ i : u_i < v_i\},\mbox{ the last index where }\ul{u}\mbox{ and }\ul{v}\mbox{ differ}.
\end{align*}

\begin{lemma}\label{lem:3factor}
 Every degree three monomial $z_{\ul{u}} z_{\ul{v}} z_{\ul{w}} \in \ol{R}_3$ can be expressed as a linear combination
  \begin{equation}\label{eq:3factor}
  z_{\ul{u}} z_{\ul{v}} z_{\ul{w}} = \sum c_{\ul{u}',\ul{v}',\ul{w}'}\cdot z_{\ul{u}'} z_{\ul{v}'} z_{\ul{w}'}
  \end{equation}
where $c_{\ul{u}', \ul{v}', \ul{w}'} \in \bb{Z}$, and the sum ranges over triples $(\ul{u}'<\ul{v}'<\ul{w}')$ with $f(\ul{v}',\ul{w}') < l(\ul{u}',\ul{v}')$ and $(|\ul{u}'|,|\ul{v}'|,|\ul{w}'|)\preceq(|\ul{u}|,|\ul{v}|,|\ul{w}|)$.
\end{lemma}

\begin{proof}
 Arguing as in the proof of Corollary~\ref{cor:incr-ui} we choose a monomial $m = z_{\ul{u}} z_{\ul{v}} z_{\ul{w}}$ of smallest lexicographic rank for which the conclusion of the lemma fails (assuming there exists one). Corollary~\ref{cor:incr-ui} (and its proof) implies that we may assume $\ul{u} < \ul{v} < \ul{w}$. Since $m$ can't be expressed as in (\ref{eq:3factor}), we must have
 \[ f(\ul{v},\ul{w}) = f \geq l = l(\ul{u},\ul{v}).\]
We rewrite $z_{\ul{v}}$ using (\ref{eq:re-express-zv}) and consider any $\ul{v}\neq\ul{v}'\in\mc{P}(\ul{a})$ with $|\ul{v}'| = |\ul{v}|$. One of the following holds:
\begin{itemize}
 \item $v_i' < u_i$ for some $i=1,\cdots,n$. Using (\ref{eq:SL}) we get
 \[ z_{\ul{u}} z_{\ul{v}'} z_{\ul{w}} = z_{\min(\ul{u},\ul{v}')} z_{\max(\ul{u},\ul{v}')} z_{\ul{w}},\]
 and since $|\min(\ul{u},\ul{v}')| < |\ul{u}|$ the right hand side has smaller lexicographic rank and hence can be expressed as in (\ref{eq:3factor}).
 \item $v_i' > w_i$ for some $i=1,\cdots,n$. We argue as in the previous case, using $z_{\ul{v}'} z_{\ul{w}} = z_{\min(\ul{v}',\ul{w})} z_{\max(\ul{v}',\ul{w})}$.
 \item $u_i\leq v_i'\leq w_i$ for all $i=1,\cdots,n$. Since $|\ul{v}| = |\ul{v}'|$ and $\ul{v}\neq\ul{v}'$ there exists $i$ such that $v_i>v'_i$. Since $v'_i \geq u_i$ we get that $i \leq l$, and since $w_i\geq v_i>v'_i$ it follows that 
 \[f(\ul{v}',\ul{w}) \leq i\leq l.\]
 By a similar argument, there exists $j$ such that $v_j<v'_j$ which yields
 \[l(\ul{u},\ul{v}') \geq j \geq f.\]
 Since $i\neq j$ and $f\geq l$ we get that $f(\ul{v}',\ul{w}) < l(\ul{u},\ul{v}')$, hence $z_{\ul{u}} z_{\ul{v}'} z_{\ul{w}}$ can be expressed as in (\ref{eq:3factor}).
\end{itemize}
The above analysis allows us to express $m$ as in (\ref{eq:3factor}), contradicting the assumption and thus concluding the proof.
\end{proof}

A similar argument shows the following.

\begin{lemma}\label{lem:2factor}
 Every degree two monomial $z_{\ul{u}} z_{\ul{v}} \in \ol{R}_2$ can be expressed as a linear combination
  \begin{equation}\label{eq:2factor}
  z_{\ul{u}} z_{\ul{v}} = \sum c_{\ul{u}',\ul{v}'}\cdot z_{\ul{u}'} z_{\ul{v}'}
  \end{equation}
where $c_{\ul{u}', \ul{v}'} \in \bb{Z}$, and the sum ranges over pairs $(\ul{u}'<\ul{v}')$ with $f(\ul{u}',\ul{v}') < l(\ul{0},\ul{u}')$, $f(\ul{v}',\ul{a}) < l(\ul{u}',\ul{v}')$, $(|\ul{u}'|,|\ul{v}'|)\preceq(|\ul{u}|,|\ul{v}|)$.
\end{lemma}

\begin{proof}
 Left to the reader.
\end{proof}

\begin{proof}[Proof of Theorem~\ref{thm:basis-Rbar}]
\begin{subequations}
 It is enough to prove that the monomials $m_{\gamma}$ span $\ol{R}$, since their number is equal to the vector space dimension of $\ol{R}$ by \eqref{eq:deg-segre}. Using the arguments in Corollary~\ref{cor:incr-ui}, and Lemmas~\ref{lem:3factor} and~\ref{lem:2factor}, it follows that $\ol{R}_r$ is spanned by monomials $m=z_{\ul{u}^1} \cdots z_{\ul{u}^r}$ where $\ul{u}^1 < \cdots < \ul{u}^r$ and
\begin{equation}\label{eq:order-f-l}
 f(\ul{u}^i,\ul{u}^{i+1}) < l(\ul{u}^{i-1},\ul{u}^i)\mbox{ for all }i=1,\cdots,r,
\end{equation}
where we make the convention that $\ul{u}^0 = \ul{0}$ and $\ul{u}^{r+1} = \ul{a}$. We thus fix such a monomial $m$ and let $d_i = |\ul{u}^i|$ for $i=0,\cdots,r+1$. The goal is to construct a path $\gamma\in\mc{L}(\ul{a})$ with $\gamma(d_i) = \ul{u}^i$ and $\op{Desc}(\gamma) = \{d_1,\cdots,d_r\}$, which will show that $m=m_{\gamma}$.

Observe that for every $\ul{v} \geq \ul{0}$ there exists a unique path $\gamma_{\ul{v}} \in \mc{L}(\ul{v})$ with no descents: we make the convention $v_{n+1}=1$ and for every $0\leq k\leq |\ul{v}|$ we consider the unique index $0\leq i\leq n$ with
\[v_1+\cdots+v_i \leq k < v_1+\cdots+v_{i+1}\] 
and let
\begin{equation}\label{eq:def-gamma-v}
\gamma_{\ul{v}}(k) = (v_1,\cdots,v_i,k-v_1-\cdots-v_i,0,\cdots,0).
\end{equation}
We define $\gamma\in\mc{L}(\ul{a})$ by concatenating the paths $\gamma_{\ul{u}^{i+1}-\ul{u}^i}$ in the following sense: we let $\gamma(|\ul{a}|) = \ul{a}$, while for $i=0,\cdots,r$ and $d_i\leq k<d_{i+1}$ we define
\begin{equation}\label{eq:def-gamma}
\gamma(k) = \ul{u}^i + \gamma_{\ul{u}^{i+1}-\ul{u}^i}(k-d_i).
\end{equation}
It is clear that $\gamma(d_i) = \ul{u}^i$ for all $i=0,\cdots,r+1$, and that $\op{Desc}(\gamma) \subseteq \{d_1,\cdots,d_r\}$. The fact that each of $d_1,\cdots,d_r$ appears in $\op{Desc}(\gamma)$ follows from (\ref{eq:order-f-l}), since for each $i=1,\cdots,r$ we have
\[s(\gamma)_{d_i} = l(\ul{0},\ul{u}^i-\ul{u}^{i-1}) = l(\ul{u}^{i-1},\ul{u}^i) \qquad \mbox{ and } \qquad s(\gamma)_{d_i+1} = f(\ul{0},\ul{u}^{i+1}-\ul{u}^i) = f(\ul{u}^i,\ul{u}^{i+1}).\qedhere \]
\end{subequations}
\end{proof}

\subsection{The straightening algorithm}

Based on the arguments in the proof of Theorem~\ref{thm:basis-Rbar} we can devise a recursive algorithm for expressing any monomial $m$ of the form (\ref{eq:some-m}) as a linear combination of elements of the standard basis $\mc{B}(\ul{a})$.

Consider a monomial
\[
m = z_{\ul{u}^1} \cdots z_{\ul{u}^r} \in \R_r,\mbox{ with }\ul{u}^i\in\mc{P}(\ul{a})\mbox{ and }|\ul{u}^1| \leq \cdots \leq |\ul{u}^r|,
\]
and make the convention that $\ul{u}^0 = \ul{0}$ and $\ul{u}^{r+1} = \ul{a}$. If $\ul{u}^1 =  \ul{0}$ or $\ul{u}^r = \ul{a}$ then $m=0$. Otherwise,

{\bf Step $1$:} Is there an index $i=1,\cdots,r-1$ such that $\ul{u}^i \not\leq \ul{u}^{i+1}$? If there is one, consider any such index and use (\ref{eq:SL}) to replace $m$ by $m'$ where
\[m' = z_{\ul{u}^1} \cdots z_{\ul{u}^{i-1}}\cdot z_{\min(\ul{u}^{i},\ul{u}^{i+1})}\cdot z_{\max(\ul{u}^{i},\ul{u}^{i+1})}  \cdot z_{\ul{u}^{i+2}}\cdots z_{\ul{u}^r} \]
Observe that $\lexrk(m') < \lexrk(m)$, and run the algorithm on $m'$.

If there is no such index $i$ then go to Step~2.

~

{\bf Step $2$:} Is there an index $i=1,\cdots,r$ such that $f(\ul{u}^{i},\ul{u}^{i+1}) \geq l(\ul{u}^{i-1},\ul{u}^{i})$? If there is one, consider the maximal such index $i$ and use (\ref{eq:re-express-zv}) to write
\[m = - \sum_{\substack{\ul{u}^i \neq \ul{v} \in \mc{P}(\ul{a}) \\ |\ul{v}| = |\ul{u}^i|}} m(\ul{u}^i \rightsquigarrow \ul{v}),\]
where $m(\ul{u}^i \rightsquigarrow \ul{v})$ is obtained from $m$ by replacing $\ul{u}^i$ with $\ul{v}$. Run the algorithm on each $m(\ul{u}^i \rightsquigarrow \ul{v})$.

If there is no such index $i$ then go to Step~3.

~

{\bf Step $3$:} Construct the path $\gamma\in\mc{L}(\ul{a})$ via (\ref{eq:def-gamma}) by concatenating the paths $\gamma_{\ul{u}^{i+1}-\ul{u}^i}$ defined through (\ref{eq:def-gamma-v}) and observe that $m = m_{\gamma}$.

\begin{example}\label{ex:1111}
 Suppose that $n=4$ and $\ul{a} = (1,1,1,1)$, and consider the monomial
 \[m = z_{0,0,0,1} \cdot z_{0,1,0,1} \cdot z_{1,1,0,1}.\]
 Set $\ul{u}^0 = (0,0,0,0)$, $\ul{u}^1 = (0,0,0,1)$, $\ul{u}^2 = (0,1,0,1)$, $\ul{u}^3 = (1,1,0,1)$, $\ul{u}^4 = (1,1,1,1)$.
 
 Since $(0,0,0,1) < (0,1,0,1) < (1,1,0,1)$ we can proceed directly to Step 2. We have
 \[f( (0,0,0,1) , (0,1,0,1)) = 2\mbox{ and } l( (0,0,0,0),(0,0,0,1) ) = 4,\]
 \[f( (0,1,0,1) , (1,1,0,1)) = 1\mbox{ and } l( (0,0,0,1),(0,1,0,1) ) = 2,\]
 \[f( (1,1,0,1) , (1,1,1,1)) = 3\mbox{ and } l( (0,1,0,1),(1,1,0,1) ) = 1,\]
 so $i=3$ is the only index for which $f(\ul{u}^{i},\ul{u}^{i+1}) \geq l(\ul{u}^{i-1},\ul{u}^{i})$. We write
 \[ z_{1,1,0,1} = - z_{1,1,1,0} - z_{1,0,1,1} - z_{0,1,1,1}.\]
 If $\ul{v} = (1,1,1,0)$ then $z_{0,1,0,1}\cdot z_{\ul{v}} \overset{(\ref{eq:SL})}{=} z_{0,1,0,0}\cdot z_{1,1,1,1}=0$, so $m((1,1,0,1) \rightsquigarrow \ul{v}) = 0$. By a similar reasoning, if $\ul{v} = (1,0,1,1)$ then $m((1,1,0,1) \rightsquigarrow \ul{v}) = 0$. It follows that 
 \[ m = - m((1,1,0,1) \rightsquigarrow (0,1,1,1)) = - z_{0,0,0,1} \cdot z_{0,1,0,1} \cdot z_{0,1,1,1}.\]
 We now consider 
\[ m' = z_{0,0,0,1} \cdot z_{0,1,0,1} \cdot z_{0,1,1,1},\]
and set $\ul{v}^0 = (0,0,0,0)$, $\ul{v}^1 = (0,0,0,1)$, $\ul{v}^2 = (0,1,0,1)$, $\ul{v}^3 = (0,1,1,1)$, $\ul{v}^4 = (1,1,1,1)$. We find that $i=2$ is the only index for which $f(\ul{v}^{i},\ul{v}^{i+1}) \geq l(\ul{v}^{i-1},\ul{v}^{i})$ and write
\[ z_{0,1,0,1} = - z_{1,1,0,0} - z_{1,0,1,0} - z_{1,0,0,1} - z_{0,1,1,0} - z_{0,0,1,1}.\]
If $\ul{w} = (1,1,0,0)$, $(1,0,1,0)$ or $(0,1,1,0)$ then using (\ref{eq:SL}) we get $z_{0,0,0,1}\cdot z_{\ul{w}} = 0$ and therefore $m'((0,1,0,1) \rightsquigarrow \ul{w}) = 0$. If $\ul{w} = (1,0,0,1)$ then $z_{\ul{w}} \cdot z_{0,1,1,1} = z_{0,0,0,1} \cdot z_{1,1,1,1} = 0$, so $m'((0,1,0,1) \rightsquigarrow \ul{w}) = 0$. It follows that
\[ m' = - m'((0,1,0,1) \rightsquigarrow (0,0,1,1)) = - z_{0,0,0,1} \cdot z_{0,0,1,1} \cdot z_{0,1,1,1}.\]
Putting everything together, and writing $\gamma$ for the unique path with $s(\gamma) = (4,3,2,1)$ we obtain
\[ 
m = - m' = z_{0,0,0,1} \cdot z_{0,0,1,1} \cdot z_{0,1,1,1} = m_{\gamma}. \qedhere
\]
\end{example}

\section{(Non-)vanishing of syzygies for a Segre product of copies of $\bb{P}^1$}\label{sec:non-vanishing-P1}
In this section we study the syzygies of $X=\Seg(1^n)$, the Segre product of $n$ copies of $\bb{P}^1$. In particular, we prove that the vanishing result of Theorem~\ref{thm:vanishing} is sharp. More precisely, it follows from Theorem~\ref{thm:vanishing} and Lemma~\ref{lem:Paaa} that
\[K_{p,q}(1^n)=0\rm{ for }p<2^{q+1}-2-q,\rm{ and }q=0,\cdots,n-1.\]
Using the fact that $\Seg(1^n)$ is arithmetically Gorenstein, it follows that
\[K_{p,q}(1^n)=0\rm{ for }p>2^n-2^{n-q}-q,\rm{ and }q=0,\cdots,n-1.\]
The main result of this section is

\begin{theorem}\label{thm:nonvanishing}
 With the notation as above, we have
\begin{equation}\label{eq:nonvanishing}
K_{p,q}(1^n)\neq 0\rm{ for }2^{q+1}-2-q\leq p\leq 2^n-2^{n-q}-q,\rm{ and }q=0,\cdots,n-1. 
\end{equation}
\end{theorem}

The remaining part of this section will focus on proving Theorem~\ref{thm:nonvanishing}. Recall that $\R$ (resp. $\S$) is the quotient of $R$ (resp. $S$) by the regular sequence $x_0,x_1,\cdots,x_n$ defined in (\ref{eq:defxr}), and that by (\ref{eq:reduction-syz}) we have
 \[K_{p,q}(1^n)=K_{p,q}^{\S}(\R).\]
 It suffices then to prove the corresponding non-vanishing for the syzygies of $\R$ as an $\S$-module, for which we employ the strategy of \cite{EEL}, expanding on the arguments from Section~\ref{sec:non-vanishing-Kp1}. We prove (\ref{eq:nonvanishing}) for fixed $q$, by induction on $n=q+1,q+2,\cdots$. When $n=q+1$ we only need to consider $p=2^{q+1}-2-q$ and the conclusion follows from the fact that $\Seg(1^q)$ has projective dimension $2^{q+1}-q-1$, regularity $q$, and is arithmetically Gorenstein (see Remark~\ref{rem:a-divides-q}). When $q=1$, (\ref{eq:nonvanishing}) is a special case of Theorem~\ref{thm:nonvanishing-Kp1}. By duality we conclude that (\ref{eq:nonvanishing}) also holds for $q=n-2$. We will therefore assume from now on that $n\geq q+3$.
 
 By (\ref{eq:kpq-specialization}), the non-vanishing $K_{p,q}(1^{n-1})\neq 0$ implies that $K_{p,q}(1^n)\neq 0$, so instead of (\ref{eq:nonvanishing}) it is enough to prove that
\begin{equation}\label{eq:nonvanishing'}
K_{p,q}(1^n)\neq 0\rm{ for }2^{n-1}-2^{n-1-q}-q< p\leq 2^n-2^{n-q}-q.
\end{equation}
 Recall from (\ref{eq:standard-basis}) the definition of the standard basis $\mc{B}=\mc{B}(1^n)$ of $\R$. We write $\mc{B}_q\subset \mc{B}$ for the subset consisting of elements of degree $q$, for $q=0,\cdots,n-1$. Just as in Theorem~\ref{thm:nonvanishing-Kp1}, we will prove the non-vanishing of $K_{p,q}(1^n)$ by exhibiting a Koszul cycle which is not a boundary, and having the form
 \begin{equation}\label{eq:koscycle}
c = (z_{\ul{v}^1}\wedge z_{\ul{v}^2}\wedge\cdots\wedge z_{\ul{v}^p})\oo m \in \bw^p \S_1 \oo \R_q,
 \end{equation}
 where $\ul{v}^i\in \mc{L}(1^n)$ and $m\in \mc{B}_q$. We consider the unique path $\gamma\in\mc{L}(1^n)$ for which
 \[s(\gamma) = (n,n-1,\cdots,n-q+1,1,2,\cdots,n-q),\]
 and take
 \begin{equation}\label{eq:defm}
m = m_{\gamma} = z_{(0^{n-1},1)} \cdot z_{(0^{n-2},1^2)} \cdots z_{(0^{n-q},1^q)}.
 \end{equation}
We consider the set
\begin{equation}\label{eq:annihilators-zm}
\begin{split}
 \mc{A} &= \left(\{ \ul{v} \in \mc{L}(1^n) : v_i = 0 \mbox{ for some }i\geq n-q+1\} \cup \{(0^{n-q},1^q)\}\right)\\
&\qquad \qquad \setminus \{(0,1^i,0^{n-i-1}):i=0,\cdots,q\}.
\end{split}
\end{equation} 
We show that the elements of $\mc{Z}(\mc{A})$ annihilate $m$. To do so we start with an important observation.

Recall that $\R$ has a $\bb{Z}^2$-grading given by degree and rank. We write $\R_{d,r}$ for the span of monomials of degree $d$ and rank $r$. 

\begin{lemma}\label{lem:small-rank}
We have that
 \[\R_{d,r} = 0 \mbox{ if }r<{d+1 \choose 2}.\]
\end{lemma}

\begin{proof}
 It suffices to check that a standard basis element $m_{\nu}$ of degree $d$ has rank $r\geq {d+1 \choose 2}$. To do so, note that
 \[\rk(m_{\nu}) \overset{(\ref{eq:m-gamma})}{=} \sum_{t \in \op{Desc}(\nu)} |\nu(t)| = \sum_{t \in \op{Desc}(\nu)} t \geq 1 + 2 + \cdots + d = {d+1 \choose 2}.\]
 where the inequality follows from the fact that $\op{Desc}(\nu)$ consists of $d$ distinct elements from $\{1,\cdots,n\}$.
\end{proof}

\begin{lemma}\label{lem:annihilators}
We have that $|\mc{A}| = 2^n - 2^{n-q}-q$ and
\[m \cdot z_{\ul{v}}=0\mbox{ for all }\ul{v}\in\mc{A}.\] 
\end{lemma}

\begin{proof} The formula for $|\mc{A}|$ follows from (\ref{eq:annihilators-zm}). Suppose first that $\ul{v}\in\mc{A}$ is such that $v_i = 0$ for some $i\geq n-q+1$. Using (\ref{eq:SL}) we get
\[z_{(0^{n-q},1^q)} \cdot z_{\ul{v}} = z_{\ul{v}'}\cdot z_{\ul{v}''},\]
where $\ul{v}' = \min((0^{n-q},1^q),\ul{v})$ satisfies $|\ul{v}'| < q$. This proves that
\[
m \cdot z_{\ul{v}} = z_{(0^{n-1},1)} \cdot z_{(0^{n-2},1^2)} \cdots z_{(0^{n-q+1},1^{q-1})} \cdot z_{\ul{v}'} \cdot z_{\ul{v}''} = 0
\]
by applying Lemma~\ref{lem:small-rank} since $m' = z_{(0^{n-1},1)} \cdot z_{(0^{n-2},1^2)} \cdots z_{(0^{n-q+1},1^{q-1})} \cdot z_{\ul{v}'}$ satisfies
\[
\deg(m') = q\mbox{ and }\rk(m') = 1 + 2 + \cdots + (q-1) + |\ul{v}'| < {q+1 \choose 2}.
\]
If instead we have that $\ul{v} = (0^{n-q},1^q)$ then $m\cdot z_{\ul{v}}$ has degree $q+1$ and rank ${q+1 \choose 2} + q < {q+2 \choose 2}$, so the conclusion $m\cdot z_{\ul{v}} = 0$ follows from Lemma~\ref{lem:small-rank}.
\end{proof}

We next consider the set
\begin{equation}\label{eq:divisors-zm}
 \mc{D} = \{\ul{v} = (v_1,0^{n-q-1},v_{n-q+1},\cdots,v_n) : \ul{v} \neq \ul{0} \mbox{ and }\ul{v}\neq(1,0^{n-q-1},1^{q})\}.
\end{equation} 
 We have that
 \[\mc{D} \subset \mc{A} \mbox{ and }|\mc{D}| = 2^{q+1}-2 \leq 2^{n-1}-2^{n-1-q}-q.\]
The desired non-vanishing result now follows once we prove the following.

\begin{lemma}\label{lem:partial-nonvanishing}
 Let $2^{q+1}-2 \leq p \leq 2^n - 2^{n-q} - q$ and consider any set $\mc{V}$ with $|\mc{V}| = p$ and $\mc{D}\subseteq\mc{V}\subseteq\mc{A}$. If we choose an ordering $\ul{v}^1,\cdots,\ul{v}^p$ of the elements of $\mc{V}$ and define $c$ via (\ref{eq:koscycle}) then $c$ represents a non-zero cohomology class in $K_{p,q}^{\S}(\R)$.
\end{lemma} 
 
Indeed, the assumption that $n\geq q+3$ guarantees that $2^{q+1}-2 \leq 2^{n-1}-2^{n-1-q}-q$. Lemma~\ref{lem:partial-nonvanishing} then implies that $K_{p,q}^{\S}(\R) \neq 0$ for $2^{n-1}-2^{n-1-q}-q \leq p \leq 2^n - 2^{n-q}-q$ and since we know by induction that $K_{p,q}^{\S}(\R) \neq 0$ for $2^{q+1}-2-q\leq p\leq 2^{n-1}-2^{n-1-q}-q$, the non-vanishing (\ref{eq:nonvanishing}) follows.
 
\begin{proof}[Proof of Lemma~\ref{lem:partial-nonvanishing}]
 It follows from Lemma~\ref{lem:annihilators} that $m\cdot z_{\ul{v}} = 0$ for all $\ul{v}\in\mc{V}$, so $c$ is a Koszul cycle, hence we only need to verify that $c$ is not a boundary. We first introduce some more notation. Consider $x\in\R$ and $m_{\nu} \in \mc{B}$: we say that  \defi{$x$ divides $m_{\nu}$}, and write $x|m_{\nu}$, if there exists $x'\in\R$ such that $m_{\nu}$ appears with non-zero coefficient in the expression of $x\cdot x'$ as a linear combination of the elements in $\mc{B}$. Example~\ref{ex:1111} shows for instance that for $n=4$ we have
\[z_{0,1,0,1} \mbox{ and }z_{1,1,0,1} \mbox{ divide }z_{0,0,0,1} \cdot z_{0,0,1,1} \cdot z_{0,1,1,1}.\]

Since $\S_1 = \R_1$, the set $\mc{B}_1$ is a basis of $\S_1$. It consists of
\[\mc{B}_1 = \mc{Z}(\mc{L}(1^n) \setminus \{(1^i,0^{n-i}): i=0,\cdots,n\}).\]
This induces for every $p,q$ a basis $\mc{B}^{\wedge p}_q$ of $\bw^p \S_1 \oo \R_q$ given by
\[\mc{B}^{\wedge p}_q = \{ z_1 \wedge \cdots \wedge z_p \oo m_{\nu} : \{z_1,\cdots,z_p\}\mbox{ is a subset of }\mc{B}_1\mbox{ of size }p\mbox{ and }m_{\nu}\in\mc{B}_q\}.\]
Note that $\mc{Z}(\mc{D}) \subseteq \mc{B}_1$, but $\mc{Z}(\mc{A}) \not\subseteq \mc{B}_1$, so the Koszul cycle $c$ is not necessarily an element of $\mc{B}^{\wedge p}_q$. Nevertheless, there exists scalars $\a_{z_1,\cdots,z_p}$ such that $c$ is expressed as a linear combination
\begin{equation}\label{eq:c-as-linear-combination}
 c = \sum_{z_1,\cdots,z_p \in \mc{B}_1} \a_{z_1,\cdots,z_p} \cdot z_1 \wedge \cdots \wedge z_p \oo m.
\end{equation}
This expression is obtained from (\ref{eq:koscycle}) by replacing any element $z_{\ul{v}^i}\notin\mc{B}_1$ with the corresponding linear combination of elements in $\mc{B}_1$. Since $\mc{Z}(\mc{D}) \subseteq \mc{B}_1$, none of the elements $z_{\ul{v}^i}$ with $\ul{v}^i\in\mc{D}$ are affected by this replacement, hence we have
\begin{equation}\label{eq:z1zp-contain-D}
 \mbox{if }\a_{z_1,\cdots,z_p} \neq 0\mbox{ then }\{z_1,\cdots,z_p\} \supseteq \mc{Z}(\mc{D}).
\end{equation}
We won't be concerned with the precise value of the scalars $\a_{z_1,\cdots,z_p}$, but we note that all of the terms in \eqref{eq:c-as-linear-combination} have the monomial $m$ as the factor in $\R_q$. It will also be important to note that using Lemma~\ref{lem:basis-R1} we get that $\mc{Z}(\mc{A})$ is a linearly independent set. As a consequence $\mc{Z}(\mc{V})$ is linearly independent and thus $c \neq 0$.

If $c$ was a boundary, then $c = \partial(b)$ for some $b\in\bw^{p+1} \S_1 \oo \R_{q-1}$, where $\partial$ denotes the Koszul differential. We express $b$ as a linear combination of elements in $\mc{B}^{\wedge(p+1)}_{q-1}$ and observe that this linear combination involves, with non-zero coefficient, terms of the form
\[z_1 \wedge \cdots \wedge z_{p+1} \oo m_{\nu} \in \mc{B}^{\wedge(p+1)}_{q-1}\]
where $z_1,\cdots,z_p$ are such that the coefficient $\a_{z_1,\cdots,z_p}$ in \eqref{eq:c-as-linear-combination} is non-zero, and where the expression of $z_{p+1} \cdot m_{\nu}$ as a linear combination of elements in $\mc{B}_q$ involves $m$ appearing with non-zero coefficient, i.e. $z_{p+1}$ divides $m$. By Lemma~\ref{lem:divisors-of-m} below, if $z_{p+1}$ divides $m$ then $z_{p+1}\in\mc{Z}(\mc{D})$, which combined with \eqref{eq:z1zp-contain-D} implies that $z_{p+1} \in \{z_1,\cdots,z_p\}$, contradicting the fact that $z_1 \wedge \cdots \wedge z_{p+1} \oo m_{\nu}$ was an element of $\mc{B}^{\wedge(p+1)}_{q-1}$.
\end{proof}

To conclude the proof of Theorem~\ref{thm:nonvanishing} we prove the following.

\begin{lemma}\label{lem:divisors-of-m}
 If $z_{\ul{u}}\in\mc{B}_1$ divides $m$ then $\ul{u}\in\mc{D}$.
\end{lemma}

\begin{proof}
 Consider elements $z_{\ul{u}}\in\mc{B}_1$ and $m_{\nu} \in \mc{B}_{q-1}$ and assume that we have an expression of $z_{\ul{u}}\cdot m_{\nu} $ in terms of the basis $\mc{B}_q$:
 \[z_{\ul{u}}\cdot m_{\nu} = \a\cdot m + \cdots \mbox{ for some non-zero scalar }\a.\]
Using the grading of $\R$ by rank, we obtain
\[ \rk(z_{\ul{u}}) + \rk(m_{\nu}) = \rk(m) = {q+1 \choose 2}.\]
We write $m_{\nu} = z_{\ul{u}^1} \cdots z_{\ul{u}^{q-1}}$ with $1\leq |\ul{u}^1| < \cdots < |\ul{u}^{q-1}|\leq n-1$, and rewrite the above equality as
\begin{equation}\label{eq:rank-monomial}
 |\ul{u}| +  |\ul{u}^1| + \cdots + |\ul{u}^{q-1}| = 1 + 2 + \cdots + q.
\end{equation}
Let $i_0$ denote the maximal index for which $|\ul{u}^{i_0}| = i_0$ (if it does not exist, set $i_0=0$), so that $|\ul{u}^j| \geq j+1$ for $j=i_0+1,\cdots,q-1$. If $|\ul{u}| \leq i_0$ then it follows from Lemma~\ref{lem:small-rank} that
\[ z_{\ul{u}} \cdot z_{\ul{u}^1} \cdots z_{\ul{u}^{i_0}} = 0,\mbox{ thus }z_{\ul{u}}\cdot m_{\nu}=0,\]
contradicting the assumption that $\a\neq 0$. If $|\ul{u}| \geq i_0+2$ then
\[|\ul{u}| +  |\ul{u}^1| + \cdots + |\ul{u}^{q-1}| \geq (i_0+2) + 1 + \cdots + i_0 + (i_0+2) + \cdots + q > 1+2+\cdots+q,\]
contradicting \eqref{eq:rank-monomial}. It follows that $|\ul{u}| = i_0+1$ and $|\ul{u}^j| = j+1$ for $j=i_0+1,\cdots,q-1$.

If $\ul{u}^{i_0} \not\leq\ul{u}$ then we can use \eqref{eq:SL} to write $z_{\ul{u}^{i_0}} \cdot z_{\ul{u}} = z_{\ul{u}'} \cdot z_{\ul{u}''}$ with $|\ul{u}'| < i_0$, which implies by Lemma~\ref{lem:small-rank} that
\[ z_{\ul{u}^1} \cdots z_{\ul{u}^{i_0-1}} \cdot z_{\ul{u}'} = 0,\mbox{ and therefore }z_{\ul{u}^1} \cdots z_{\ul{u}^{i_0}} \cdot z_{\ul{u}} = z_{\ul{u}^1} \cdots z_{\ul{u}^{i_0-1}} \cdot z_{\ul{u}'}\cdot z_{\ul{u}''}=0.\]
This yields $z_{\ul{u}}\cdot m_{\nu}=0$, contradicting the assumption that $\a\neq 0$. A similar contradiction is obtained if $\ul{u} \not\leq \ul{u}^{i_0+1}$. These arguments show more generally that for every collection $\ul{v}^1,\cdots,\ul{v}^q \in \mc{L}(1^n)$ with $|\ul{v}^i| = i$ for all $i=1,\cdots,q$ we have
\[z_{\ul{v}^1} \cdots z_{\ul{v}^q} = 0 \mbox{ unless }\ul{v}^1 < \cdots < \ul{v}^q.\]
For such monomials, we can thus modify the Straightening Algorithm to consist of a series of operations involving only the following modification of {\bf Step $2$}:

{\bf Step $2'$:} If $m' = z_{\ul{v}^1} \cdots z_{\ul{v}^q}$ with $\ul{v}^1 < \cdots < \ul{v}^q$, $|\ul{v}^i| = i$ and $m' \not\in \mc{B}_q$ then consider the maximal index $i$ for which $f(\ul{v}^{i},\ul{v}^{i+1}) \geq l(\ul{v}^{i-1},\ul{v}^{i})$. Replace $m'$ by
\begin{equation}\label{eq:new-straightening}
- \sum_{\substack{\ul{v}^i \neq \ul{v} \in \mc{P}(1^n) \\ |\ul{v}| = |\ul{v}^i| \\ \ul{v}^{i-1} < \ul{v} < \ul{v}^{i+1}}} m'(\ul{v}^i \rightsquigarrow \ul{v})
\end{equation}
and apply {\bf Step $2'$} to each of the monomials $m'(\ul{v}^i \rightsquigarrow \ul{v})$: note that this and every subsequent application of {\bf Step $2'$} will preserve the factors $z_{\ul{v}^{i+1}} \cdots z_{\ul{v}^q}$.

We apply the modified version of the straightening algorithm to the monomial
\[m' = z_{\ul{v}^1} \cdots z_{\ul{v}^q},\mbox{ where }\ul{v}^j = \begin{cases}
 \ul{u}^j & \mbox{if }j\leq i_0; \\
 \ul{u} & \mbox{if }j = i_0+1; \\
 \ul{u}^{j-1} & \mbox{if }j\geq i_0+2.
\end{cases}
\]
If $\ul{v}^q$ remains unchanged in the process, then $\ul{v}^q = (0^{n-q},1^q)$ (since $m' = \a\cdot m + \cdots$) and in particular $\ul{u} \leq (0^{n-q},1^q)$, which yields $\ul{u}\in\mc{D}$ as desired.

Otherwise $\ul{v}^q$ must be changed at the first iteration of {\bf Step $2'$}: we have $i=q$ in \eqref{eq:new-straightening} and $\ul{v} = (0^{n-q},1^q)$ must appear among the terms in \eqref{eq:new-straightening}. If $i_0 \leq q-2$ then we obtain
\[ \ul{u} = \ul{v}^{i_0+1} \leq \ul{v}^{q-1} < \ul{v}\]
and conclude as before that $\ul{u}\in\mc{D}$.

We are left with the case $i_0 = q-1$, when $\ul{u} = \ul{v}^q$ and $\ul{u}^{q-1} = \ul{v}^{q-1}$. Let $l = l(\ul{u}^{q-1},\ul{u}) \leq f(\ul{u},(1^n))$, and note that since $|\ul{u}^{q-1}| = q-1$ and $|\ul{u}| = q$, $l$ is the unique index where $\ul{u}$ differs from $\ul{u}^{q-1}$, and in particular $u_l = 1$. Moreover, every vector $\ul{u}\neq\ul{v}\in\mc{L}(1^n)$ with $\ul{v} > \ul{u}^{q-1}$ and $|\ul{v}| = q$ is obtained from $\ul{u}$ by replacing $u_l$ with $0$ and changing one of the zero entries in $\ul{u}$ to $1$. Since $\ul{v} = (0^{n-q},1^q)$ is one of these vectors, it follows that $l\leq n-q$, and moreover $l$ is the unique index in $\{1,\cdots,n-q\}$ for which $u_l = 1$. If $l>1$ then $f(\ul{u},1^n) = 1 < l$ which is a contradiction. It follows that $l=1$, $u_2 = \cdots = u_{n-q} = 0$ and thus $\ul{u}\in\mc{D}$.
\end{proof}

\section{Examples}\label{sec:examples}

To illustrate the (non-)vanishing behavior in (\ref{eq:Kpq-allP1}) we record below the graded Betti tables of Segre embeddings of products of $\bb{P}^1$ that we were able to compute using Macaulay2:

\begin{myverbbox}[\small]{\Segtwo}
       0 1
total: 1 1
    0: 1 .
    1: . 1
\end{myverbbox}
\begin{myverbbox}[\small]{\Segthree}
       0 1  2 3 4
total: 1 9 16 9 1
    0: 1 .  . . .
    1: . 9 16 9 .
    2: . .  . . 1
\end{myverbbox}
\begin{myverbbox}[\small]{\Segfour}
       0  1   2   3    4    5    6    7   8   9 10 11
total: 1 55 320 891 1436 1375 1375 1436 891 320 55  1
       1  .   .   .    .    .    .    .   .   .  .  .
       . 55 320 891 1408 1183  192   28   .   .  .  .
       .  .   .   .   28  192 1183 1408 891 320 55  .
       .  .   .   .    .    .    .    .   .   .  .  1
\end{myverbbox}

\begin{center}
\begin{tabular}{c@{\hskip 2in}c}
 \Segtwo & \Segthree \\[3ex]
 Betti table of $\Seg(1,1)$ & Betti table of $\Seg(1,1,1)$
\end{tabular}
\end{center}

\bigskip

\begin{center}
\Segfour

\bigskip

Betti table of $\Seg(1,1,1,1)$
\end{center}

Besides the examples above and the one in the introduction, the only Betti table for a Segre product of $n\geq 3$ factors that we were able to determine computationally is for $\Seg(2,1,1)$:
\begin{myverbbox}[\small]{\Segtwooneone}
       0  1  2   3  4  5  6 7
total: 1 24 84 126 94 46 21 4
    0: 1  .  .   .  .  .  . .
    1: . 24 84 126 84 10  . .
    2: .  .  .   . 10 36 21 4
\end{myverbbox}

\begin{center}
\Segtwooneone
\end{center}
where you can check directly that statements (\ref{eq:Kp1-not-0}) and (\ref{eq:Kpq=0}) in our Main Theorem are sharp.

The first case when our results are not sharp arises for the Segre embedding of $\bb{P}^2\times\bb{P}^2\times\bb{P}^2$ into $\bb{P}^{26}$, where we have additional vanishing as explained below.

\begin{lemma}\label{lem:K101=0}
$K_{10,1}(2,2,2)=K_{11,1}(2,2,2) = 0$.
\end{lemma}

\begin{proof}
  First, consider the resolution of $\Seg(2,2,1)$. Since $K_{8,1}(2,2,1)$ is a sum of polynomial representations of $\GL(3) \times \GL(3) \times \GL(2)$ of degree $9$, and it has dimension $8$, the only possibility is that
  \[
    K_{8,1}(2,2,1) = S_{3,3,3}(\bb{C}^3) \otimes S_{3,3,3}(\bb{C}^3) \otimes S_{8,1}(\bb{C}^2).
  \]
  There is a nonzero equivariant map
  \[
    K_{8,1}(2,2,1) \to K_{7,1}(2,2,1) \otimes \bb{C}^3 \otimes \bb{C}^3 \otimes \bb{C}^2,
  \]
  which implies, by the Pieri rule, that the $63$-dimensional space $K_{7,1}(2,2,1)$ contains a summand of the form
  \[
    S_{3,3,2}(\bb{C}^3) \otimes S_{3,3,2}(\bb{C}^3) \otimes S_\lambda(\bb{C}^2)
  \]
  where $\lambda$ is either $(8)$ or $(7,1)$. The former gives a representation of dimension $81$, which is too large, so we conclude that $\lambda=(7,1)$. Since this representation has dimension $63$, we conclude that
  \[
    K_{7,1}(2,2,1) = S_{3,3,2}(\bb{C}^3) \otimes S_{3,3,2}(\bb{C}^3) \otimes S_{7,1}(\bb{C}^2).
  \]

  Now, we try to compute $K_{11,1}(2,2,2)$ using Theorem~\ref{thm:pushfwd-resolution}. We work over the space of lines in $V_3$ with $\dim(V_3)=3$ and 
  \[
    \mc{M} = \bigoplus_{d \ge 0} \Sym^d(V_1) \otimes \Sym^d(V_2) \otimes \Sym^d(Q).
  \]
  The only terms in Theorem~\ref{thm:pushfwd-resolution} that can possibly contribute are
\begin{align*}
  H^1\left( \bw^{12}(V_1 \otimes V_2 \otimes \mc{R}) \right), \qquad
  H^0\left( \mc{K}_{i,1}(\mc{M}) \otimes \bw^{11-i}(V_1 \otimes V_2 \otimes \mc{R})\right).
\end{align*}
We focus on which representations of $\GL(V_3)$ these produce. The first term contributes groups of the form $H^1(S_\lambda(\mc{R}))$ where $\lambda_1 \le \dim(V_1 \otimes V_2) = 9$, and such that $(\lambda_1 - 1, 1, \lambda_2)$ is weakly decreasing by Bott's theorem. But no such partition exists, so there is no contribution.

For the second term, we get representations of the form $S_{(\lambda_1,\lambda_2,11-i)}(\bb{C}^3)$ where $\lambda_1+\lambda_2=i+1$. If $i \le 6$, there is no such partition. If $i=7,8$, the above argument shows that $\lambda_2=1$, so again there is no such partition. A similar argument applies to show that $K_{10,1}(2,2,2)=0$.
\end{proof}

Since $P(2,2,2;3)=12$, the conclusion (\ref{eq:Kpq=0}) of our main theorem shows that $K_{p,3}(2,2,2)=0$ if $p<9$. Since the coordinate ring of $\Seg(2,2,2)$ is Gorenstein and it has projective dimension $20$ and regularity $4$, the vanishing $K_{10,1}(2,2,2)=K_{11,1}(2,2,2) = 0$ implies that
\[K_{9,3}(2,2,2)=K_{10,3}(2,2,2) = 0,\]
so we get more vanishing than predicted in the third row of the Betti table. 

Unlike (\ref{eq:Kpq=0}) which fails to be tight for $\Seg(2,2,2)$, it follows from Lemma~\ref{lem:K101=0} that (\ref{eq:Kp1-not-0}) still gives the precise non-vanishing range in this example, namely $K_{p,1}(2,2,2) \neq 0$ for $1\leq p\leq 9$.

\section*{Acknowledgments} 
Experiments with the computer algebra software Macaulay2 \cite{M2} have provided numerous valuable insights. Raicu acknowledges the support of the Alfred P. Sloan Foundation, and of the National Science Foundation Grant No. 1600765. Sam was partially supported by National Science Foundation Grant No. 1500069.



\begin{bibdiv}
  \begin{biblist}

\bib{bruns-ASL}{article}{
   author={Bruns, Winfried},
   title={Additions to the theory of algebras with straightening law},
   conference={
      title={Commutative algebra},
      address={Berkeley, CA},
      date={1987},
   },
   book={
      series={Math. Sci. Res. Inst. Publ.},
      volume={15},
      publisher={Springer, New York},
   },
   date={1989},
   pages={111--138},
   review={\MR{1015515 (90h:13012)}},
}

\bib{hodge}{book}{
   author={De Concini, Corrado},
   author={Eisenbud, David},
   author={Procesi, Claudio},
   title={Hodge algebras},
   series={Ast\'erisque},
   volume={91},
   note={With a French summary},
   publisher={Soci\'et\'e Math\'ematique de France, Paris},
   date={1982},
   pages={87},
   review={\MR{680936}},
}

\bib{EEL}{article}{
   author={Ein, Lawrence},
   author={Erman, Daniel},
   author={Lazarsfeld, Robert},
   title={A quick proof of nonvanishing for asymptotic syzygies},
   journal={Algebr. Geom.},
   volume={3},
   date={2016},
   number={2},
   pages={211--222},
   issn={2214-2584},
   review={\MR{3477954}},
   doi={10.14231/AG-2016-010},
}

\bib{eisenbud-ASL}{article}{
   author={Eisenbud, David},
   title={Introduction to algebras with straightening laws},
   conference={
      title={Ring theory and algebra, III},
      address={Proc. Third Conf., Univ. Oklahoma, Norman, Okla.},
      date={1979},
   },
   book={
      series={Lecture Notes in Pure and Appl. Math.},
      volume={55},
      publisher={Dekker, New York},
   },
   date={1980},
   pages={243--268},
   review={\MR{584614}},
}

\bib{M2}{article}{
          author = {Grayson, Daniel R.},
          author = {Stillman, Michael E.},
          title = {Macaulay 2, a software system for research
                   in algebraic geometry},
          journal = {Available at \url{http://www.math.uiuc.edu/Macaulay2/}}
        }

\bib{rubei-P1}{article}{
   author={Rubei, Elena},
   title={On syzygies of Segre embeddings},
   journal={Proc. Amer. Math. Soc.},
   volume={130},
   date={2002},
   number={12},
   pages={3483--3493},
   issn={0002-9939},
   review={\MR{1918824}},
   doi={10.1090/S0002-9939-02-06597-8},
}

\bib{rubei-PN}{article}{
   author={Rubei, Elena},
   title={Resolutions of Segre embeddings of projective spaces of any
   dimension},
   journal={J. Pure Appl. Algebra},
   volume={208},
   date={2007},
   number={1},
   pages={29--37},
   issn={0022-4049},
   review={\MR{2269826}},
   doi={10.1016/j.jpaa.2005.11.010},
}

\bib{weyman}{book}{
   author={Weyman, Jerzy},
   title={Cohomology of vector bundles and syzygies},
   series={Cambridge Tracts in Mathematics},
   volume={149},
   publisher={Cambridge University Press},
   place={Cambridge},
   date={2003},
   pages={xiv+371},
   isbn={0-521-62197-6},
   review={\MR{1988690 (2004d:13020)}},
   doi={10.1017/CBO9780511546556},
}

  \end{biblist}
\end{bibdiv}
\end{document}